\documentclass{article}
\usepackage{amsmath}
\usepackage{amssymb}
\usepackage{bm}
\usepackage{graphicx}
\usepackage{pifont}
\usepackage{epsfig}
\usepackage{amscd}
\usepackage{latexsym}
\usepackage{amsfonts}
\usepackage{amsthm}
\usepackage{epsfig}
\usepackage{xcolor}
\usepackage{natbib}
\usepackage{stmaryrd}

\usepackage{etoolbox}
\newcounter{bibcount}
\makeatletter
\patchcmd{\@lbibitem}{\item[}{\item[\hfil\stepcounter{bibcount}{\thebibcount.}}{}{}
\renewcommand\NAT@bibsetup%
   [1]{\setlength{\leftmargin}{\bibhang}\setlength{\itemindent}{-\parindent}%
       \setlength{\itemsep}{\bibsep}\setlength{\parsep}{\z@}}
\makeatother

\newtheorem{theorem}{Theorem}

\newtheorem{definition}{Definition}

\newtheorem{lemma}{Lemma}[section]

\newtheorem{proposition}{Proposition}
\newtheorem{remark}{Remark}[section]

\numberwithin{equation}{section}

\newcommand\bes{\begin{eqnarray}}
\newcommand\ees{\end{eqnarray}}
\newcommand\bess{\begin{eqnarray*}}
\newcommand\eess{\end{eqnarray*}}

\newcommand\vep{{\varepsilon}}

\newcommand{\ve}{\varepsilon}

\newcommand{\bolddot}{{\boldsymbol{\cdot}}}

\newcommand\cL{{\mathcal L}}

\newcommand\BbbR{{\mathbb R}}

\newcommand\cG{{\mathcal G}}
\newcommand{\cM}{{\mathcal M}}

\newcommand\bX{{\mathbf X}}
\newcommand\bY{{\mathbf Y}}
\newcommand\fX{{\frak X}}
\newcommand\fB{{\frak B}}
\newcommand\fY{{\frak Y}}

\newcommand\BbbP{{\mathbb P}}

\newcommand\cF{{\mathcal F}}
\newcommand\cP{{\mathcal P}}

\begin{document}

\title{Higher-Order Regularity
of the Free Boundary  in \\ the Inverse First-Passage  Problem}
\author{Xinfu Chen\thanks{Department of Mathematics, University of Pittsburgh, xinfu@pitt.edu} \and 
John Chadam\thanks{Department of Mathematics, University of Pittsburgh, chadam@pitt.edu} \and 
David Saunders\thanks{Department of Statistics and Actuarial Science, University of Waterloo, dsaunders@uwaterloo.ca}}
\maketitle

\begin{abstract}{\small
Consider the  inverse first-passage  problem:   Given
a diffusion process $\{\fX_{t}\}_{t\geqslant 0}$ on a probability space $(\Omega,\cF,\BbbP)$  and a survival probability function $p$ on $[0,\infty)$, find a boundary, $x=b(t)$,  such that
$p$ is
 the survival probability that $\fX$ does not fall below $b$, i.e., for each $t\geqslant 0$,
$p(t)= \BbbP( \{\omega\in\Omega\;|\; \fX_s(\omega) \geqslant b(s),\ \forall\, s\in(0,t)\})$. In earlier work, we analyzed 
viscosity solutions of a related variational inequality, and showed that they provided the only upper semi-continuous (usc)
solutions of the inverse problem. We furthermore proved weak regularity (continuity) of the boundary $b$ under additional 
assumptions on $p$. The purpose of this paper is to study higher-order regularity properties of the solution of the 
inverse first-passage problem. 
In particular, we show that when $p$ is smooth and has negative slope, the viscosity solution, and therefore also the unique usc solution  of the inverse problem, is smooth. Consequently,  the viscosity solution furnishes a unique classical solution to the free boundary problem  associated with the inverse first-passage  problem.}
\end{abstract}
\date{\today}
\maketitle

\section{Introduction}

\subsection{The First-Passage Problem and its Inverse}
Let  $\fX=\{\fX_t\}_{t\geqslant 0}$ be the solution to the stochastic differential equation
\bess d\fX_t =\mu(\fX_t,t) \,dt +\sigma(\fX_t,t)\, d\fB_t\quad\forall\, t>0,\eess
where $\{\fB_t\}_{t\geqslant 0}$ is a standard Brownian motion defined on a  probability space $(\Omega,\BbbP)$  and $\mu,\sigma$ are smooth bounded functions with $\inf_{\BbbR\times[0,\infty)}\sigma>0$.
The {\it boundary crossing}, or {\it first-passage} problem for the process $\fX$ concerns the following:
\medskip

\begin{enumerate}
\item[{\bf 1.}] {\bf The  Forward Problem: } Given a  function $b: (0,\infty)\to[-\infty,\infty)$, compute the survival probability, $\cP[b]$,  that $\fX$ does not fall below  $b$, i.e. evaluate
    \bes\label{1p} \cP[b](t) := \BbbP( \{\omega\in\Omega\;|\; \fX_s(\omega) \geqslant b(s)\ \forall\, s\in(0,t)\})\quad\forall\, t\geqslant 0.\ees

    \item[{\bf 2.}]{\bf The Inverse  Problem: } Given a survival probability
    $p$, find a  barrier $b$ such that $\cP[b]=p$.
            \end{enumerate}

\medskip
The forward  problem is classical and the subject of a large literature. According to  \citet{ZuccaSacerdote}, the inverse problem was first suggested by A.N. Shiryaev during a Banach Centre meeting, for the case where 
$\fX$ is a Brownian motion, and the first-passage distribution is exponential, i.e. $p(t) = e^{-\lambda t}$ for some $\lambda > 0$. 
\citet{DudleyGutmann} proved the existence of a stopping time for $\fX$ with a given law. \citet{Anulova} demonstrated the existence of a 
stopping time of the form $\tau = \inf\{ t\;|\; (\fX_{t},t)\in B\}$ for a closed set $B \subset \{ (x,t) \;|\;  x\in [-\infty,\infty], t \in [0,\infty]\}$ 
with the properties that if $(x,t)\in B$ then $(-x,t) \in B$, and if  
$x\geq 0$ and $(x,t) \in B$, then $[x,\infty] \times \{t\} \subseteq B$. Defining 
$b(t) = \inf\{ x\geqslant 0: (x,t)\in B\}: [0,\infty]\to[0,\infty]$, then $b$ satisfies the two-sided version of the inverse problem, 
$p(t) = \BbbP( \{\omega\in\Omega\;|\; -b(s) \leqslant \fX_{s}(\omega) \leqslant b(s)\ \forall\, s\in(0,t)\}),\, \forall\, t\geqslant 0$.

\medskip
In the 2000's, the inverse problem became the subject of renewed interest due to applications in financial mathematics. 
In particular, \citet{AvellanedaZhu} formulated the one-sided inverse problem given above as a free boundary 
problem for the Kolmogorov forward equation associated with~$\fX$, and discussed its numerical solution. Other numerical approaches and applications to credit risk were studied 
by \citet{HullWhiteTwo}, \citet{IscoeKreinin}, and \citet{HuangTian}. 
In~\citet{CCCS1}, we presented a rigorous mathematical analysis of the free boundary problem (formulated as a variational inequality), first demonstrating the existence of a unique viscosity solution 
to the problem, and analyzing integral equations satisfied by the boundary $b$, and its asymptotic behaviour for small $t$. 
In~\citet{CCCS2}, we proved that the boundary $b$ arising from the variational inequality does indeed solve the 
probabilistic formulation of the inverse problem. We also studied weak regularity properties of the free boundary 
$b$. In particular, we proved the following (see \citet{CCCS2}, Proposition 6): 
\begin{proposition}\label{WeakRegularityProposition} 
Suppose that $\fX$ is standard Brownian motion, started at 0 (i.e. $\mu\equiv 0$, $\sigma\equiv 1$, and $\fX_{0} = 0$), 
and that $p$ is continuous with $p(0)=1$. Define: 
\begin{equation} 
L(p,T_{1},T_{2}) := \inf_{T_{1} \leqslant s < t \leqslant T_{2}} \frac{p(s)-p(t)}{t-s}, \quad \forall 0\leqslant T_{1} < T_{2}.
\nonumber
\end{equation} 
\begin{enumerate}
\item If $L(p,T_{1},T_{2}) > 0$ for some positive $T_{1}, T_{2}$ with $T_{1} < T_{2}$, then $b$ is continuous on $(T_{1},T_{2})$. 
\item Assume that $L(p,0,T) > 0$ for every $T > 0$. Then $b\in C([0,\infty))$.
\end{enumerate}
\end{proposition}
In particular, if $p(t) = e^{-\lambda t}$, then $L(p,0,T) = e^{-\lambda T}>0$, yielding continuity of the boundary in the problem 
posed by Shiryaev. 

\medskip
{\bf The purpose of this paper is to study higher-order regularity properties of the boundary $b$ in the one-sided inverse problem.}

\medskip
Several other papers have also studied the inverse problem. \citet{EkstromJanson} show that the solution of the 
inverse first-passage problem is the same as the solution of a related optimal stopping problem, and present an analysis 
of an associated integral equation for the stopping boundary $b$. Integral equations related to the problem are discussed 
in~\citet{Pe1} and \citet{PeskirShiryaev}. \citet{Abundo} studied the small-time behaviour of the boundary $b$. A rigorous 
construction of the boundary based on a discretization procedure was recently presented in~\citet{Potiron}.

\begin{remark}
Traditionally, the forward problem is studied for upper-semi-continuous (usc) $b$ and the conventional survival probability, $\hat p$, is defined in term of the
first crossing time, $\hat\tau$, by
\bess \qquad \hat\tau(\omega) := \inf\{t>0\;|\; \fX_t(\omega)\leqslant b(t)\}, \qquad 
\hat p(t):=\BbbP(\{\omega\in\Omega\;|\;\hat\tau(\omega) >t\}).\eess
Here, the  survival probability, $p=\cP[b]$ in (\ref{1p}), is defined by
\bess \tau(\omega) :=\inf\{t>0\;|\; \fX_t(\omega)<b(t)\}, \qquad p(t) := \BbbP(\{\omega\in\Omega\;|\; \tau(\omega)\geqslant t\}).
\eess
It is shown in \cite{CCCS2} that $\cP[b]$ is well-defined for each $b$. In addition, define  $b^*$
 and $b_-^*$ by
\bess  b^*(t):=\max\Big\{b(t),\;\varlimsup_{s\to t} b(s)\Big\},\quad b^*_-(t):=\varlimsup_{s\nearrow t} b(s).\label{3a.bstar}\eess
Then (i) $\cP[b]=\cP[b^*]$, (ii) when $b=b^*$,  $\tau=\hat\tau$ almost surely,  and (iii) when $b^*=b_-^*$, $\cP[b]\in C((0,\infty))$.
Since $\cP[b]=\cP[b^*]$, it is convenient for the inverse problem to restrict the search to $b$  in the class of usc functions, i.e., those $b$ that satisfy $b=b^*$.  In \cite{CCCS2} it is  shown that for every $b\in P_0$, where
\bes
 P_0 := \{ p\in C([0,\infty))\; \;|\;\;
p(0)=1\geqslant p(s)\geqslant p(t)>0\quad\forall\, t> s\geqslant 0\},\label{P0}
 \ees
the inverse problem admits a unique usc solution.
\end{remark}


\subsection{The Free Boundary Problems}
We introduce differential  operators $\cL$ and $\cL_1$ defined by
 \bess \cL \phi := \partial_t \phi -\tfrac12 \partial_x(\sigma^2 \partial_x\phi)+\mu \partial_x\phi,
 \qquad\cL_1 \phi:= \partial_t \phi-\tfrac12 \partial_{xx}^2 (\sigma^2\phi) +\partial_x(\mu\phi).\eess
 The
 survival distribution, $w$, and survival density,  $u$, are defined by
\bess w(x,t): =\BbbP( \tau \geqslant t, \fX_t>x),\qquad  u(x,t):= -\partial_x w(x,t). \eess 
We denote the  distribution  of $\fX_t$ by $w_0$ and its density by $u_0$:
\bess w_0(x,t):=\BbbP(\fX_t>  x),\qquad u_0(x) dx :=  \BbbP(\fX_0 \in(x,x+dx)).\eess

When $b$ is smooth, one can show that $(b,w,p)$ satisfies
\bes \label{1w} \left\{ \begin{array}{ll} \cL w=0 \quad &\hbox{for \ } x>b(t), t>0,
\\  \partial_x w(x,t) =0 &\hbox{for \ } x\leqslant b(t),t>0,
\\ w(x,0) =w_0(x,0) &\hbox{for \ } x\in\BbbR,\ \  t=0,
\\ p(t)=w(b(t),t)   &\hbox{for \ } x= b(t), t>0 ,\end{array}\right.
\ees
and  $(b,u,p)$ satisfies
\bes\label{1u} \left\{ \begin{array}{ll} \cL_1 u=0 \quad &\hbox{for \ } x>b(t), t>0,
\\  u(x,t) =0 &\hbox{for \ } x\leqslant b(t),t>0,
\\ u(x,0) =u_0(x) &\hbox{for \ } x\in\BbbR,\ \  t=0,
\\ 2\dot p(t)= - \sigma^2\partial_x u|_{x=b(t)+}  &\hbox{for \ } x=b(t), t>0 .\end{array}\right.
\ees
 Note that (\ref{1w}) and (\ref{1u}) are equivalent in the class of smooth functions via the transformation
\bess u(x,t)=-\partial_x w(x,t),\qquad w(x,t)=\int_x^\infty u(y,t)dy\qquad\forall\, x\in\BbbR,t\geqslant 0.\eess

 When $b$ is given and  regular, say, Lipschitz continuous, the forward problem can be easily handled  by first solving the initial--boundary value problem consisting of the first three equations in (\ref{1w})  and then evaluating $p$ from the last equation in  (\ref{1w}).
For the inverse problem, both (\ref{1w}) and (\ref{1u}) are free boundary problems since the domain
$Q_b:=\{(x,t)\;|\; x>b(t),t>0\}$, where the equations, $\cL w=0$ and $\cL_1 u=0$, are satisfied, is a priori unknown.
So far, there is little known concerning the  existence,  uniqueness, and regularity of classical solutions of the free boundary problems.
Here by a classical solution, $(b,w)$,  of (\ref{1w}) we mean that $w-w_0\in C(\BbbR\times[0,\infty)), \partial_x w\in C(\BbbR\times(0,\infty)), \partial_t w,\partial^2_{xx} w\in C(Q_b)$, and each equation in (\ref{1w}) is satisfied; similarly, by a classical solution, $(b,u)$, of (\ref{1u}), we mean that
$u+w_{0x}\in C(\BbbR\times[0,\infty)), \partial_t u, \partial_{xx}^2u\in C(Q_b)$, and each equation in (\ref{1u}) is satisfied.  
In this paper, we investigate the well-posedness of the free boundary problem and the smoothness of the free boundary.
\medskip

\subsection{The Weak Formulation}

In \cite{CCCS1}, viscosity solutions for the inverse problem,   based on   the variational inequality
\bess \max\{ \cL w, w-p\}  =0\quad\hbox{in \ }\BbbR\times(0,\infty),\eess
 are  introduced.  It is shown that for any given probability distribution $p$ on $[0,\infty)$, there exists a unique viscosity solution.
This was followed up in~\citet{CCCS2} in which it was shown that the viscosity solution of the variational inequality 
gives the solution of the (probabilistic) inverse problem. 
 For easy reference,  we quote the relevant results.\bigskip


\begin{definition}\label{def1}
Let $p\in P_0$  be given where $P_0$ is as in {\rm(\ref{P0})}.    A {\rm viscosity solution} of the inverse problem associated with $p$ is a function $b$ defined by
\bes b(t):=\inf\;\{ x\in\BbbR\;|\; w(x,t)<p(t)\},\quad\forall\, t>0, \label{2.ba}\ees
 provided that $w$
    has   the following properties:

\begin{enumerate}
\item  $w\in C(\BbbR\times(0,\infty))$, $\lim_{t\searrow 0}\|w(\bolddot,t)
-\BbbP(\fX_t>\,\boldsymbol{\cdot}\,)\|_{L^\infty(\BbbR)}=0$;

\item   $0\leqslant w\leqslant p$ in
$\BbbR\times(0,\infty)$ and
$\cL w =0
$ in the set
$\{(x,t)\;|\; t>0, w(x,t)<p(t)\}$;

\item  If for a smooth
$\varphi$, $x\in\BbbR$ and $ t>\delta>0$,  the function $\varphi-w$ attains its  local minimum on $[x-\delta,x+\delta]\times [t-\delta,t]$ at
$(x,t)$,
 then  $\cL \varphi(x,t)\leqslant 0$.
\end{enumerate}
\end{definition}
\bigskip

One can verify that if $(b,w)$ (or $(b,u)$) is a classical solution of the free boundary problem (\ref{1w}) (or (\ref{1u})), then $b$ is a viscosity solution of the inverse problem associated with $p$.\bigskip



\begin{proposition}[{\bf Well-posedness of the Inverse Problem \cite{CCCS1,CCCS2}}]
Let $p\in P_0$ be given.

\begin{enumerate}
\item  \cite{CCCS1}: There exists a unique viscosity solution, $b$, of the inverse problem associated with $p$.

\item \cite{CCCS2}: The viscosity solution  is a usc solution of the inverse problem, i.e. $b=b^*$ and $\cP[b]=p$.

\item \cite{CCCS2}: There exists a unique usc  solution of the inverse problem associated with $p$.
\end{enumerate}
\end{proposition}
\bigskip

It is clear now that the viscosity solution is the right choice for the inverse problem.
  For convenience, in the sequel,  we shall call $b$, $(b,w)$, $(b,u)$, or $(b,w,u)$, the solution or the viscosity solution of the inverse problem, where $b$ is the viscosity solution boundary, $w$  is the viscosity solution for the survival distribution, and $u=-\partial_x w$ is  the viscosity solution for the survival density of the inverse problem associated with $p$.
We shall also call the curve  $x=b(t)$ the free boundary.
\medskip

\subsection{The Main Result: Higher Order Regularity} While the work of \cite{CCCS1} and \citet{CCCS2}  solves the inverse problem, 
and presents a basic study of weak regularity, 
here we make a detailed study of the regularity of the free boundary. The main result of this paper is the following, where 
$\llbracket\alpha+\frac12\rrbracket$  denotes  the integer part of $\alpha+\frac12$.\bigskip

\begin{theorem}[{\bf Regularity of the Free Boundary}] Let $p\in P_0$ be given and
  $(b,w,u)$ be the {\rm(viscosity)}  solution of the inverse problem associated with $p$.
Assume that
 $p\in C^1([0,\infty)), \dot p<0 \hbox{\ on\  }[0,\infty)$, and for some $\delta>0$,  either
\bes\label{1.up} \hbox{\rm  (i) \ }
 u_0=0\hbox{  on  }(-\infty,0], \ u_0 \in C^1([0,\delta]),\  u_0'(0+)>0,\qquad\hbox{\rm  or \  (ii) \ }
\ddot p \in L^1((0,\delta)).\ees

\begin{enumerate} \item Then $(b,w)$ and $(b,u)$ are classical solutions of {\rm(\ref{1w})} and {\rm(\ref{1u})} respectively
with \boldmath$b\in C^{1/2}((0,\infty))$\unboldmath.

 \medskip

\item If, in addition, for $\alpha > \frac12$ not an integer, $(t_1,t_2)\subset(0,\infty)$ one has $p\in C^{\alpha+\frac12}((t_1,t_2))$, 
then \boldmath$b\in C^\alpha((t_1,t_2))$\unboldmath.

\medskip

\item Finally, if for $\alpha\geqslant \frac12$ not an integer one has $p\in C^{\alpha+\frac12}([0,\infty))$, 
$u_0=0$ on $(-\infty,0]$, $u_0\in C^{2\alpha}([0,\delta])$, and $(u_0,p)$ satisfy all compatibility conditions up to order 
$\llbracket\alpha+\frac12\rrbracket$, including in particular the compatibility condition  $2\dot p(0)=-{\sigma^2(0,0)\; u_0'(0+)}$ 
when $\alpha\in[\frac12,\frac32)$, then \boldmath$b(0)=0$ {\bf and} $b\in C^\alpha([0,\infty))$\unboldmath.
\end{enumerate}

%
%
%
\label{mainth}\end{theorem}
\bigskip

\begin{remark} \label{re1} {\sl To derive the compatibility conditions,  we consider,
 for simplicity, the special case $\sigma\equiv \sqrt2$ and $\mu\equiv0$.
Set $U(x,t)=u(x+b(t),t)$.  Then
\bes\label{1U} \left\{\begin{array}{ll}
U_t=U_{xx}+\dot b\, U_x\qquad &\hbox{in \ }(0,\infty)\times(0,\infty),\medskip\\   U(0,\cdot)=0,
   U_x(0,\cdot)=-\dot p(\cdot)&\hbox{on \ }\{0\}\times(0,\infty),\medskip
   \\ U(\cdot,0)=u_0(\cdot)&\hbox{on \ } [0,\infty)\times\{0\}.\end{array}\right.\ees
    The $k$-th order compatibility condition is  the $(k-1)$-th  order derivative of  $\dot p(t)=-U_x(0,t)$  at  $t=0$ (with differentiation of  $U$ in time  being replaced by differentiation in space by    $U_{t}= U_{xx}+\dot b U_{x}$):
  \bess \frac{d^k p(t)}{dt^k}\Big|_{t=0} = -
  \frac{\partial^{k-1}U_x(0,t)}{\partial t^{k-1}} \Big|_{t=0} = -\frac{d^{2k-1} u_0(x)}{dx^{2k-1}} \Big|_{x=0}+ \cdots.\eess
   The $k$-th order derivative of $b$ at $t=0$ is obtained by differentiating the equation
   $\dot b(t)\dot p(t) =U_{xx}(0,t)$:
   \bess \frac{d^{k}b(t)}{dt^{k}} \Big|_{t=0} =
  \frac{\partial^{k-1}U_{xx}(0,t)}{\dot p(0)\;\partial t^{k-1}} \Big|_{t=0}+\cdots =
   \frac{1}{\dot p(0)} \frac{d^{2k}u_0(x)}{dx^{2k}}\Big|_{x=0} + \cdots.\eess
In particular, the first and second order compatibility conditions are
\bess \dot p(0)=-u_0'(0),\qquad  \ddot p(0)=- u_0'''(0) -\dot b(0) u_0''(0)\Big|_{\dot b(0)= {u_0''(0)/\dot p(0)}}.\eess
}\end{remark}

\begin{remark}{\sl
  For  $b$ to be continuous and bounded, it is necessary to assume that
$p$ is strictly decreasing as in Proposition~\ref{WeakRegularityProposition}. 
Indeed, if $p$ is a constant in an open interval, then $b=-\infty$ in that interval.

}
\end{remark}
\medskip
The main tool for the proof of Theorem~\ref{mainth} is the hodograph transformation, defined by the change of variables 
$x=X(z,t)$, the inverse of $z=u(x,t)$.
Since $u(b(t),t)=0$, we have $b(t)=X(0,t)$. 
Taking $\sigma\equiv\sqrt{2}$ to 
simplify the exposition, and proceeding formally, we can derive that 
$X$  solves quasi-linear pde:
\bes\label{1X} 
Y_t = Y_{z}^{-2} Y_{zz}+[z \mu(Y,t)]_z,\qquad \dot p(t)\;Y_z(0,t)= -1,\quad u_0(Y(z,0))=z.
\ees
  Assuming it can be shown that $X_z(\vep,\cdot)$ is positive on $[0,T]$, this system is studied on the set $\{(z,t)\,|\, z\in[0,\vep], t\in [0,T]\}$ for any fixed $T>0$ and a small  positive $\vep$ that depends on $T$. To complete the  system, we supply  the  boundary condition for $Y$ on  $\{\vep\}\times(0,T]$  by $Y_z(\vep,t)=X_z(\vep,t)$.

 \medskip
 The classical approach to the hodograph transformation 
(see, e.g.~\citet{FriedmanFreeBoundaryProblems} or~\citet{KinderlehrerStampacchia}) employs a bootstrapping strategy, 
assuming some 
initial degree of smoothness on $(p,b)$, and then using the regularity theory for~(\ref{1X}) to strengthen the 
regularity of $b$. In particular, standard results for quasi-linear equations 
(\citet{LSU, Lieberman}) can be used to derive the existence of a unique classical solution of~(\ref{1X}), and its 
regularity (including up to the boundary). The assumptions on $(p,b)$ are sufficient to reverse the hodograph transformation, 
and transfer the boundary regularity of $Y$ to $b(t)=Y(0,t)$ (and to show that indeed $Y=X$, where $X$ is defined through
$x = X(u(x,t),t)$). The results achieved through the classical approach are reviewed in Section~\ref{ClassicalHodographSection}.

\medskip
In order to prove Theorem~\ref{mainth}, we wish to employ the same strategy, but with weaker assumptions on the initial 
regularity of $(p,b)$. In doing so, we encounter two main difficulties, the first technical, and the second fundamental. The 
technical issue is that above the boundary, $u$ solves $\cL_{1} u =0$, and when $\mu\ne 0$
the operator $\cL_{1}$ has a zeroth order term. The maximum principle arguments employed in analyzing level sets in our proofs require a differential operator 
with no zeroth order.\footnote{In particular, we need that constant functions satisfy the differential equation.} 
We address
these related technical hurdles by considering $v=u/K$ for an appropriate scaling function $K$, defined as the solution of an auxiliary partial 
differential equation, such that above the boundary $\cL_{2}v=0$, where $\cL_{2} = \partial_{t} - \partial_{xx} + \nu\partial_{x}$ 
for some function $\nu$. The formal hodograph transformation then leads us to consider the partial differential equation: 
\begin{equation}
Y_t  = Y_z^{-2}\; Y_{zz}+ \nu(Y,t),\qquad  z\in (0,\vep), t\in(0,T],
\end{equation} 
together with the boundary condition $\cM Y = 0$,  where 
\begin{equation}
\cM Y = \begin{cases} Y(z,0) - X_{0}(z), & z\in [0,\vep], \\
\dot p(t)\, Y_z(0,t)+K(Y(0,t),t), & t\in(0,T], \\
Y_z(\vep,t) -X_z(\vep,t), & t\in(0,T].
\end{cases}
\end{equation}
Here, we encounter a fundamental difficulty due to the fact that we do not know a priori that $X_{z}$ is regular up to the boundary, i.e., we do not have the equation $\dot p(t) X_{z}(0,t) + K(X(0,t),t)=0$. While we can study the above problem analytically, we have not assumed the 
requisite regularity to show that $Y=X$, with $X(z,t):= \min\{ x\geqslant b(t)\;|\; v(x,t)=z\}$.\footnote{We can, however, show 
that $X$ is well-defined and regular enough inside the domain. It is the regularity for $X$ up to the boundary (because of the lack of a priori regularity of $b$)  
that is insufficient.  } This difficulty is surmounted by defining a family of perturbed equations with boundary operators 
$\cM^{h}$, $h\in\BbbR$, and making comparisons with their solutions $Y^{h}$. In particular, for small $h > 0$, we show 
that $Y^{-h}(0,\bolddot) < b < Y^{h}(0,\bolddot)$, $Y^{-h}(\vep,\bolddot) < X(\vep,\bolddot) < Y^{h}(\vep,\bolddot)$. 
Letting $h\searrow 0$, we are then able to obtain that $X\equiv Y$, $b=Y(0,\bolddot)$, and the required regularity of $b$.

\begin{remark}
For simplicity of exposition, throughout the paper we assume that $\sigma\equiv\sqrt2$. This can be done without loss of 
generality. Indeed, let
\bess \bY(x,t) = \int_0^x\frac{\sqrt 2\, dz}{\sigma(z,t)},\quad \tilde \mu(y,t)=
\frac{\sqrt 2 \mu(x,t)}{\sigma(x,t)}-\frac{\partial_x \sigma(x,t)}{\sqrt 2} -\int_0^x \frac{\sqrt 2\partial_t \sigma(z,t)}{\sigma^2(z,t)}\,dz\Big|_{x=\bX(y,t)}\eess
where $x=\bX(y,t)$ is  the inverse of $y=\bY(x,t)$. Then by It\^o's lemma,
the process $\{\fY_t\}$ defined by $\fY_t:={\mathbf Y}(\fX_t,t)$ is a diffusion process satisfying
$d\fY_t= \tilde\mu(\fY_t,t) dt +\sqrt 2\,d\fB_t$. The boundary crossing problem for $\{\fX_t\}$
with barrier $b$ is  equivalent to the boundary crossing problem for $\{\fY_t\}$ with barrier
$\tilde b(t):= \bY(b(t),t)$. In terms of the partial differential equations, this is equivalent to the change of variables
\bess y=\bY(x,t),\quad \tilde w(y,t)= w(\bX(y,t),t),\quad  \tilde u(y,t)=
  -\partial_y \tilde w(y,t) \ .\eess
{\bf We shall henceforth always assume that $\sigma\equiv\sqrt2$.}
\end{remark}

The remainder of the paper is structured as follows. In the next section, we recall a few properties of the solution of the inverse problem and prove a smoothing property of the diffusion: under  condition  (ii) of (\ref{1.up}), condition (i) is satisfied provided that the  initial time $t=0$ is shifted to $t=s$, for any $s$ outside a set of measure zero.
In Section 3 we provide an interpretation  of
 the free boundary condition $\sigma^2(b(t),t)\,u_x(b(t),t)=-2\dot p(t)$  for the viscosity solution.
 In Section 4 we present results that can be derived using the traditional approach to the hodograph 
 transformation $z=u(X(z,t),t)$. Section 5 presents the proof of Theorem~\ref{mainth}, beginning by presenting a required 
 generalization of the Hopf Boundary Point Lemma, then introducing the scaling function $K$ and the scaled survival density 
 $v$, and finally analyzing the family $Y^{h}$ of solutions to quasi-linear parabolic equations in order to derive our main 
 results.

\section{Regularity Properties of the Viscosity Solutions of the Inverse Problem}
In this section, we collect a few results concerning the regularity of the (viscosity) solution of the inverse
problem. Recall that $w_0(x,t)=\BbbP(\fX_t>x)$ and $Q_b=\{(x,t)\;|\; t>0, x>b(t)\}$.
\bigskip

\begin{lemma}\label{le2.1} Let $(b,w,u)$ be the  solution of the inverse problem associated with
$p\in P_0$. Then
\bes\label{2.3}  u\in C^\infty(Q_b), \quad   \cL_1 u=0<u \hbox{ in \ }Q_b, \quad\cL_1 u\leqslant 0\leqslant u\hbox{\ in }\BbbR\times(0,\infty),
\label{BasicUProperties}
\ees
where the inequalities above hold in the sense of distributions.

In addition, for any $T>0$, the  following holds:\medskip

 \begin{enumerate}
\item  If \  $\dot p\in  L^\infty((0,T))$, then 
 $ u+w_{0x}\in \bigcap_{\alpha\in(0,1)}  C^{\alpha,\alpha/2}(\BbbR\times[0,T])$;
 consequently, $(b,w)$ is a classical solution of the free boundary problem {\rm(\ref{1w})}
 on $\BbbR\times[0,T]$. \medskip

 \item If \   $\inf\{x\;|\;\BbbP(\fX_0\leqslant x)>0\}=0$ and $\sup_{t\in[0,T]}\dot p <0$, then $b(0)=0$ and $b\in C([0,T])$.
 \medskip

\item {} {\bf (Smoothing Property)}
If  $\dot p<0$ on $[0,T]$ and $\ddot p\in L^1((0,T))$, then for a.e.  $t\in(0,T]$,
 \bess{ \ }\quad w_t(\cdot,t)\in C^{1/2}(\BbbR),\quad u(\cdot,t) \in C(\BbbR)\cap C^{3/2}([b(t),\infty)), \quad \sigma^2(b(t),t)\, u_x(b(t)+,t)=-2\dot p(t).\eess
 \end{enumerate}
 \end{lemma}
\bigskip

 {\it Proof. }  For~(\ref{BasicUProperties}),  see~\citet{CCCS1}, Lemma 2.1, and~\citet{CCCS2}, 
 Proposition 5 (with Theorems 1 and 3 in this reference summarizing the solution of the inverse boundary crossing problem).  

\medskip
To simplify the presentation, we assume that $\mu\equiv 0$. The general case is analogous.\footnote{Unlike the assumption 
that $\sigma\equiv\sqrt{2}$, there are points in the paper where this is not the case. Hence, we remark on this explicitly 
when we take the drift to be zero.} 
 The viscosity solution for the survival distribution $w$ of the inverse problem satisfies the  variational inequality
 $\max\{\cL w, w-p\}=0$, which, according to \citet{CCCS1} (see also~\citet{FriedmanFreeBoundaryProblems}), can be approximated, as $\vep\searrow 0$,   
 by the solution of
\bess   w_t^\vep - w^\vep _{xx} =-\beta(\vep^{-1}[w^\vep-p])\hbox{\ in \ }\BbbR\times(0,\infty),
  && w^\vep(\cdot,0)=w_0(\cdot,0)\hbox{ \ on \ }\BbbR\times\{0\}\eess
where  $\beta(z)= m\cdot(\max\{0,z\})^3$ with $m= \|\dot p\|_{L^\infty((0,T))}$. It is worth mentioning that one can further regularize  $(p,w_0,\beta)$  to smooth $(p^\vep,w_0^\vep,\beta^\vep)$ so that the solution is not only smooth, but also monotonically  decreasing in $\vep$; see  \cite{CCCS1} for details. We set $u^\vep=-\partial_x w^\vep$, and 
$\cG \varphi = \cL \varphi + \beta(\vep^{-1}[\varphi(x,t)-p(t)])$.
\medskip

1. Since $\cG (p+\vep) = \dot p + m \geq 0 = \cG w^{\ve}$ and $w^{\vep}(\cdot,0) \leq p(0)$, the 
Comparison Principle yields that $w^{\vep} \leq p+\vep$. Similarly, 
$\cL(w_{0}-w^{\vep}) = \beta(\vep^{-1}[w^{\vep} - p]) \geq 0$ gives $w^{\vep} \leq w_{0}$ and $\cG w^{\vep} 
\geq 0 = \cG(0)$ yields that $w^{\vep} \geq 0$. From $w^{\vep} \leq p + \vep$, we have: 
\begin{equation}  \label{BasicBetaBound}
0\leqslant -\cL (w^\vep-w_0)=\beta([w^\vep-p]\vep^{-1}) \leqslant \beta(1)=m  \quad\hbox{in \ }\BbbR\times(0,T].
\end{equation}
A parabolic estimate~(e.g. \citet{KrylovHolder}, Lemma 8.7.1, page 122)
then implies that
\begin{equation} \label{FirstHolderEstimate}
\|w^\vep-w_{0}\|_{C^{1+\alpha,(1+\alpha)/2}(\BbbR\times[0,T])}\leqslant m C(\alpha),
\end{equation} 
for every $\alpha\in(0,1)$ where $C(\alpha)$ is a constant depending only on $\alpha$.

Let $\vep_{n} \searrow 0$, and let $\beta\in(\alpha,1)$. Then there exists a subsequence $\vep_{n_{k}}$
such that $w^{\vep_{n_{k}}} - w_{0} \to w-w_{0}$ in 
$C^{1+\alpha,(1+\alpha)/2}_{\mathrm{loc}}(\BbbR\times[0,T])$. As the limit $w$ is shown in \cite{CCCS1} to be the unique viscosity solution of the inverse problem associated with $p$,  the whole sequence in fact converges. 
In addition,
 passing to the limit in the estimate~(\ref{FirstHolderEstimate}) we see that
 $u+w_{0x}=w_{0x}-w_x \in \cap_{\alpha\in(0,1)} C^{\alpha,\alpha/2}(\BbbR\times[0,T])$.


\medskip
Note that a classical solution of (\ref{1w}) requires that the 
free boundary condition  $w_x(b(t),t)=0$ be well-defined. Since $w_x=w_{0x}-[w_{0x}+u]$ is continuous on $\BbbR\times(0,\infty)$, the equation $w_x(b(t),t)=0$ for $t>0$ is satisfied  in the classical sense.
Hence, $(b,w)$ is a classical solution of (\ref{1w}). This proves (1).

\medskip
2. This is shown in \cite{CCCS2}, Proposition 6 (see Proposition~\ref{WeakRegularityProposition} above).

\medskip

3. Using energy estimates, we will show that for each $\eta \in (0,T)$, $\int_{\eta}^{T} \int_{\BbbR} w_{tx}^{2}(x,t)dx\, dt < \infty$, 
and therefore in particular 
$w_{tx}(\cdot,t) \in L^{2}(\BbbR)$ for almost every $t\in(0,T)$, from which Sobolov Embedding yields 
$w_{t}(\cdot,t)\in C^{1/2}(\BbbR)$. Since $w_t(x,t)=\dot p(t)$ for all $x< b(t)$  and $w_t(x,t)=w_{xx}=-u_x$ for all $x>b(t)$, the last assertion of the lemma thus follows. 

\medskip
Suppose $\dot p<0$ on $[0,T]$ and $\ddot p\in L^1((0,T))$.
 Then $\dot p\in C([0,T])$, so $m:=\|\dot p\|_{L^\infty((0,T))}$ is finite. As above, we have  $0\leqslant w^\vep(x,t)\leqslant w_0$.
Using the Comparison Principle as in~\citet{CCCS1}, it can be shown that $u^{\vep}=-w^{\vep}_{x} \geq 0$. We then further have
$\cL(-w_{0x}-u^{\vep}) = \beta'(\vep^{-1}(w^{\vep}-p))\vep^{-1}u^{\vep} \geq 0$, and comparison yields $u^{\vep} \leq -w_{0x}$ on 
 $\BbbR\times(0,\infty)$. 

%

\medskip
 Let $\zeta=\zeta(x)$ be a  non-negative smooth function satisfying $\zeta(x)=0$ for $x<-M$,  
  $\zeta(x)=1$ for $x>4-M$, and $0\leqslant \zeta'(x) 
 \leqslant 1$ and $|\zeta''(x)|\leqslant 1$ for $x\in[-M,4-M]$. In the sequel, for convenience of notation, 
we consider $\beta$ as a function of $(x,t)$, being evaluated at $\vep^{-1}(w^{\vep}(x,t)-p(t))$.  
 Using the differential equation and integration by parts, one can derive the identity
\bess \frac{d}{dt} \int_\BbbR \Big(\frac12 w_t^\vep{}^2 + \dot p \beta\Big)\zeta\, dx
+\int_\BbbR \Big( w_{xt}^\vep{}^2 \zeta+ |\vep^{-1}(w^\vep_t-\dot p)^{2}\beta'|\zeta -\frac12 w_t^\vep{}^2 \zeta_{xx} \Big)dx 
= \ddot p \int_\BbbR \beta\zeta\,dx,\eess
so that 
\begin{align} 
\int_{\BbbR} (w_{xt}^{\vep})^{2}\zeta\, dx &= \ddot p \int_\BbbR \beta\zeta\,dx -
\frac{d}{dt} \int_\BbbR \Big(\frac12 w_t^\vep{}^2 + \dot p \beta\Big)\zeta\, dx - \int_{\BbbR} 
\left(   |\vep^{-1}(w^\vep_t-\dot p)^{2}\beta'|\zeta -\frac12 w_t^\vep{}^2 \zeta_{xx}      \right)\, dx. \nonumber 
\end{align} 

Multiplying both sides of the above equation by $t^{2}$, and then integrating over $(0,T)$ gives:
\begin{align}
\int_{0}^{T}\int_{\BbbR} t^{2} (w_{xt}^{\vep})^{2}\zeta\, dx\, dt&= \int_{0}^{T} \int_\BbbR t^{2} \ddot p(t) \beta\zeta\,dx -
\int_{0}^{T} t^{2} \frac{d}{dt} \int_\BbbR \Big(\frac12 w_t^\vep{}^2 + \dot p \beta\Big)\zeta\, dx\, dt 
\nonumber \\
&\;\;\;- \int_{0}^{T}\int_{\BbbR}t^{2} 
\left(   |\vep^{-1}(w^\vep_t-\dot p)^{2}\beta'|\zeta -\frac12 w_t^\vep{}^2 \zeta_{xx}      \right)\, dx\, dt \nonumber \\
&\leq \int_{0}^{T} \int_\BbbR t^{2} \ddot p(t) \beta\zeta\,dx\, dt
 -\int_{0}^{T} t^{2} \frac{d}{dt} \int_\BbbR \Big(\frac12 w_t^\vep{}^2 + \dot p \beta\Big)\zeta\, dx\, dt + \int_{0}^{T}\int_{\BbbR}
\tfrac{t^{2}}{2} w_t^\vep{}^2 \zeta_{xx} \, dx \,dt \nonumber \\
&= I_{1} + I_{2} + I_{3}. \nonumber
\end{align}

To control these terms, we make two preliminary estimates. First, since $w_0(\infty,\cdot)=0$, there exists $N>0$ such that
 $w_0(x,t) <p(t)$ for every $(x,t) \in [N,\infty)\times[0,T]$. Consequently, $\beta\equiv 0$ on $[N,\infty)\times[0,T]$, so that, for every $M>0$,
 \bes\label{BetaIntegralBound}
  \int_{-M}^\infty \beta\Big(\frac{w^\vep(x,t)-p(t)}{\vep}\Big)\, dx \leqslant (M+N)m\quad\forall\, t\in[0,T].
  \ees
Secondly, integrating  $(w_t^\vep+\beta)^2-[w^\vep_{x}(w^{\vep}_t+\beta)]_{x} + w_x^\vep{} w_{xt}^\vep = -\vep^{-1}\beta' w_x^\vep{}^2\leqslant 0$ over $\BbbR$ we obtain:
  \bess \int_\BbbR (w_t^\vep+\beta)^2 dx +\frac12\;\frac{d\;}{dt} \int_\BbbR w_x^\vep{}^2\leqslant 0.\eess
 Integrating this inequality multiplied by $2t$ over $[0,s]$, we get, for every $s\in(0,\infty)$,
  \bes\nonumber && 2\int_0^s\!\! \int_\BbbR t (w^\vep_t+\beta)^2 dxdt + {{s}} \int_\BbbR w_x^\vep{}^2(x,s) dx
  \leqslant
   \int_0^s\!\!\int_\BbbR  w_x^\vep{}^2(x,t)\,dxdt
  \\ &&\quad \leqslant \int_0^s\!\! \int_\BbbR  w_{0x}^2(x,t)\, dxdt\; \leqslant \int_0^s \!\! \max_{\BbbR} |w_{0x}(\cdot,t)|\int_\BbbR |w_{0x}(x,s)|dxdt\leqslant \frac{\sqrt s}{\sqrt{\pi }},
   \label{WPlusBetaSquaredEstimate}\ees
where we have used the fact that $\int_\BbbR |w_{0x}(x,t)|dx=1$  and  $|w_{0x}(x,t)| =\Gamma*u_0 \leqslant\sup_{z\in\BbbR} \Gamma(z,t)=(4\pi t)^{-1/2}$ where  $\Gamma(x,t)=(4\pi t)^{-1/2}e^{-x^2/4t}$.

\medskip
Returning to the estimation of $I_{j}$,  (\ref{BetaIntegralBound}) immediately gives: 
\begin{equation} 
I_{1} =\int_{0}^{T} \int_\BbbR t^{2} \ddot p(t) \beta\zeta\,dx\, dt\leq (M+N)m \int_{0}^{T} t^{2} |\ddot p(t)|\, dt. 
\end{equation}
\medskip

 For $I_{2}$, we integrate by parts:
 \begin{align}
I_{2}&= -\int_{0}^{T} t^{2} \frac{d\;}{dt} \int_\BbbR \Big(\tfrac{1}{2} w_t^\vep{}^2 + \dot p \beta\Big)\zeta\, dx\, dt \nonumber \\
&= \int_{0}^{T} t\int_\BbbR \Big(w_t^\vep{}^2 +2 \dot p \beta\Big)\zeta\, dx\, dt
-t^{2}\int_{\BbbR}\Big(\tfrac{1}{2} w_t^\vep{}^2 + \dot p \beta\Big)\zeta\, dx\Bigg|_{0}^{T} \nonumber \\
&\leqslant \int_{0}^{T} t\int_\BbbR \Big(w_t^\vep{}^2 +2 \dot p \beta\Big)\zeta\, dx\, dt + T^{2}|\dot p(T)|
\int_{\BbbR} \beta\zeta \, dx.
 \end{align} 
The second term on the right is bounded by~(\ref{BetaIntegralBound}), and the fact that $\ddot p\in L^{1}([0,T])$.
Now, using that $(w^{\vep}_t)^2\leqslant 2(w_t^{\vep}+\beta)^2+2\beta^2$:
\begin{align}
\int_{0}^{T} t\int_\BbbR \Big(w_t^\vep{}^2 +2 \dot p \beta\Big)\zeta\, dx\, dt 
&\leq 2 \left(\int_{0}^{T}\int_{\BbbR}t (w_t^{\vep}+\beta)^2 \zeta\, dx\, dt + \int_{0}^{T}\int_{\BbbR}t 
\beta(\beta+\dot p)\zeta\, dx\, dt \right) \nonumber \\
& \leq \frac{\sqrt{T}}{\pi} + (M+N)m^{2} \frac{T^{2}}{2},
\end{align}
using $\dot p < 0$, (\ref{BetaIntegralBound}), and~(\ref{WPlusBetaSquaredEstimate}).

\medskip
To control $I_{3}$, we again use $(w^{\vep}_t)^2\leqslant 2(w_t^{\vep}+\beta)^2+2\beta^2$, so that: 
\begin{align} 
I_{3} &= \int_{0}^{T}\int_{\BbbR}
\tfrac{t^{2}}{2} w_t^\vep{}^2 |\zeta_{xx}| \, dx \,dt \leq 
\int_{0}^{T} \int_{\BbbR} t^{2}((w_{t}^{\vep}+\beta)^{2} + \beta^{2}) |\zeta_{xx}| \, dx\, dt. 
\end{align}
As above: 
\begin{equation} 
\int_{0}^{T}\int_{\BbbR} t^{2}\beta^{2} |\zeta_{xx}| \,dx\, dt \leqslant m^{2}\tfrac{T^{3}}{3} \int_{\BbbR}|\zeta_{xx}|\, dx
\leqslant\tfrac{2}{3}m^{2}T^{3}.
%
\end{equation}
   
\medskip
Furthermore, from~(\ref{WPlusBetaSquaredEstimate}), we have: 
\begin{align}
\int_0^T\!\! \int_\BbbR t^{2} (w^\vep_t+\beta)^2\, dx\,dt  \leq \frac{T\sqrt{T}}{2\sqrt{\pi}}.
 \end{align}
 
%

Putting things together, we have that: 
\bess  \int_{0}^{T}\!\! \int_{4-M}^\infty t^2  w^\vep_{xt}{}^2(x,t)\, dx\,dt \leqslant C(M,T), \eess
where $C(M,T)$ is a constant depending on $M$ and $T$. 
 Sending $\vep\to 0$ we see that the above estimate also holds for $w$.
Finally, since $\dot p<0$ on $[0,T]$, we have  $b\in C((0,T])$.  For each $\eta\in(0,T)$   taking $M=4+\max_{[\eta,T]}|b|$ we obtain
 \bess \int_{\eta}^T\!\!\int_\BbbR  w_{tx}^2(x,t)\, dx\,dt <\infty.\eess
 Since $\eta$ can be arbitrarily small, we conclude that
 \bess  \int_{\BbbR} w_{tx}^{2}(x,t)\, dx <\infty \quad \hbox{for almost every \ }t\in(0,T).\eess

This completes the proof.
\qed
\medskip

\begin{remark}{\sl
We remark that  for $(b,u)$ to be a classical solution of the free boundary problem {\rm(\ref{1u})}, we need the existence of the  limit of $u_x(x,t)$, as $x\searrow b(t)$, for each $t\in(0,T]$. Here the conclusion of the third assertion in the previous lemma is not sufficient for $(b,u)$ to be a classical solution.
Thus, from an analytical viewpoint, finding classical solutions of the free boundary problem {\rm(\ref{1u})} is  much harder than that of the free boundary problem {\rm(\ref{1w})}.}
\end{remark}

\begin{remark}
Taking $\fX$ to be Brownian motion with $\fX_{0} \equiv 0$, we have the following:
\begin{enumerate} 
\item If $b(\cdot)\equiv a < 0$, then: 
\begin{gather} 
p(t) = 1 - 2\int_{|a|/\sqrt{t}}^{\infty} \varphi(x) \, dx, \quad \dot p(t) = 2 \varphi\left(\frac{|a|}{\sqrt{t}} \right)\cdot \frac{|a|}{t^{3/2}}, \nonumber \\
\ddot p(t) = |a|t^{-5/2}\varphi(|a|t^{-1/2})\left(-3 + |a|t^{-1}\right), \nonumber
\end{gather} 
where $\varphi$ is the probability density function of a standard normal random variable. Clearly (1) applies ($\dot p \in
L^{\infty}((0,T))$, and $(b,w)$ is a classical solution of~(\ref{1w}), as can be verified by direct computation). 
However, $\lim_{t\searrow 0} \dot p(t) = 0$, so that (2) and (3) do not apply (which is not surprising since otherwise we would have $b(0)=0$, contradicting the definition of $b$). 
\item For the exponential survival function $p(t) = e^{-\lambda t}$ for some $\lambda > 0$, $\dot p(t) = -\lambda p(t)$ and 
$\ddot p(t) = \lambda^{2} p(t)$, so that both (2) and (3) apply, and in particular $b(0)=0$, and the smoothing property in (3) 
of the lemma holds.
\end{enumerate}
\end{remark}

\bigskip

\section{The Free Boundary Condition For Viscosity Solutions}

We  interpret the free-boundary condition
$\partial_x(\sigma^2 u)|_{x=b(t)+}=-2\dot p(t)$  for the viscosity solution as follows.

\bigskip

\begin{lemma}\label{le1a} Let $p\in P_0$ and $(b,w,u)$ be the unique viscosity solution of the inverse problem associated with $p$. Suppose  $t>0, b(t)\in\BbbR$ and $\dot p$ is continuous at $t$. Then
for any function $\ell \in C^1([0,t])$ that satisfies $\ell(t)=b(t)$, we have:
 \bes\label{1a}\varliminf_{\underset{x > \ell(s),\  s \leqslant t}{(x,s)\to (b(t),t)}} \frac{\sigma^2(x,s) u(x,s)}{x-\ell(s)} \leqslant -2 \dot  p(t)
 \leqslant \varlimsup_{\underset{x > \ell(s),\  s \leqslant t}{(x,s) \to (b(t),t)}} \frac{\sigma^2(x,s) u(x,s)}{x-\ell(s)} .
\ees
\end{lemma}

\bigskip

\begin{remark}{\sl If we know that $b\in C^1$ and $ \partial_x u\in C(\overline{Q_b})$, then taking $\ell=b$ and using L'H{\^{o}}pital's rule we see that the two limits in {\rm(\ref{1a})} are both equal to $\partial_x(\sigma^2 u)|_{x=b(t)+}$, so {\rm(\ref{1a})} provides the free boundary condition $\partial_x(\sigma^2 u)|_{x=b(t)+}=-2\dot p(t)$.
 }\label{re1a}\end{remark}

{\it Proof. }  Suppose the first inequality in (\ref{1a})  does not hold.
 Then the first limit in (\ref{1a})  is strictly bigger than $-2\dot p(t)$ so  there exist
 small constants $m>0$ and $\delta\in(0,t]$ such that
\bess u(x,s)  \geqslant (2m-\dot p(t))(x-\ell(s))\qquad\forall\, s\in[t-\delta,t],
x\in (\ell(s),b(t)+M\delta]\eess
where $M=\|\dot{\ell}\|_{L^\infty((0,t))}$.
Consequently, for each $s\in[t-\delta,t]$ and $x\in[\ell(s),b(t)+M\delta]$,
\bes\label{1a1}\qquad w(x,s) =p(s)-\int_{-\infty}^x u(y,s)dy \leqslant p(s)-\int_{\ell(s)}^x u(y,s)dy
\leqslant p(s) - (m-\tfrac{\dot p(t)}{2})(x-\ell(s))^2.\ees
Now for any sufficiently small  positive $\vep$, consider the smooth function
\bess \phi_\vep(x,s):= p(s) -  \tfrac{1}{2}(m-\dot p(t))(x-\ell_\vep(s))^2,\qquad \ell_\vep(s):=\ell(s)-(\vep-(t-s)).\eess
 We compare $w$ and $\phi_\vep$ in the set
\bess Q_\vep:=\{ (x,s)\;|\; s\in(t-\vep,t], x\in(\ell_\vep(s), \ell(s)+\sqrt\vep)\}.\eess
We claim that the minimum of $\phi_\vep-w$ on $\overline{Q_\vep}$ is negative and is  attained at some point $(x_\vep,t_\vep)\in Q_\vep$.
First of all,
 $(b(t),t)\in Q_\vep$ and
$ \phi_\vep(b(t),t)-w( b(t),t) = -\tfrac{1}{2}(m-\dot p(t)) \vep^2 <0.$

Next we show that $\phi_\vep-w\geqslant 0$ on the parabolic boundary of $Q_\vep$, 
\begin{equation} 
\partial_{p} Q_{\vep} = \{ (x,s)\;|\; s = t-\vep, x\in [\ell(t-\vep), \ell(t-\vep)+\sqrt\vep]\} \cup \{ (x,s)\;|\; s\in (t-\vep,t], x\in \{\ell_{\vep}(s),\ell(s)+\sqrt\vep \}\}.\nonumber
\end{equation}
When $x=\ell_\vep(s)$, $\phi_\vep(x,s)=p(s)$ and $\phi_\vep-w\geqslant 0$.
For small enough $\vep$, on the remainder of the parabolic boundary of $Q_\vep$, we can verify that $x\in[\ell(s),b(t)+M\delta]$ so that  we can use (\ref{1a1}) to derive
\bess \phi_\vep-w &\geqslant &
 [m-\tfrac{\dot p(t)}{2}] \;[x-\ell(s)]^2 - \tfrac{1}{2}[m-\dot p(t)]\; [x- \ell_\vep(s)]^2
 \\ &=&  \tfrac{m}{2} [x-\ell(s)]^2 + \tfrac{1}{2}[m-\dot p(t)]\;[  2x-\ell(s)-\ell_\vep(s)]\;[\ell_\vep(s)-\ell(s)].\eess
On the lower part of the parabolic boundary of $Q_\vep$, we have $s=t-\vep$ so $\ell_\vep(s)=\ell(s)$ and
$\phi_\vep\geqslant w$.
Finally, on the right  lateral boundary of $Q_\vep$, we have  $x=\ell(s)+\sqrt\vep$ and $0\leqslant \ell(s)-\ell_\vep(s)\leqslant\vep$  so
\bess \phi_\vep-w \geqslant  \tfrac{m}{2} \vep -  \tfrac{1}{2}[m-\dot p(t)] [2\sqrt \vep+\vep] \vep. \eess
Thus, $\phi_\vep-w\geqslant 0$ on the parabolic boundary of $Q_\vep$ provided that $\vep$ is sufficiently small.

Hence, for every small positive $\vep$, there exists $(x_\vep,t_\vep)\in Q_\vep$ such that $\phi_\vep-w$ attains at $(x_\vep,t_\vep)$ the minimum of $\phi_\vep-w$ over $\overline{ Q_\vep}$.
Now by the definition of $w$ as a viscosity solution (see, e.g.~\citet{CCCS1}), we have $\cL \phi_\vep(x_\vep,t_\vep)\leqslant 0$. However, we can calculate (with $\cL=\partial_s-\partial_{xx}^2+\mu \partial_x$) 
\bess \cL \phi_\vep(x_\vep,t_\vep) &=& \dot p(t_\vep) +[m-\dot p(t) ]\;[  (x_\vep-\ell_\vep(t_\vep))(\dot\ell(t_\vep)-1-\mu(x_\vep,t_\vep))+1 ]
\\ &\geqslant&  m + [\dot p(t_\vep)-\dot p(t)] -[m-\dot p(t)]\;  [\vep+\sqrt\vep]\; (M+1+|\mu(x_\vep,t_\vep)|) .\eess
The last quantity is positive if we take $\vep$ sufficiently small. Thus we obtain a contradiction, and that  the first inequality in (\ref{1a}) holds.

\bigskip

Now we prove the second inequality in (\ref{1a}).
 Since $u\geqslant 0$,  the second inequality in (\ref{1a})  is trivially true when $\dot p(t)=0$.
Hence, we consider the case $\dot p(t)<0$.
 Suppose the second inequality in (\ref{1a}) does not hold. Then the second  limit in (\ref{1a})   is strictly less than $-2\dot p(t)$ so  there exist
 small constants $m\in(0,-\dot p(t)/2)$ and $\delta\in(0,t]$ such that
\bess u(x,s)  \leqslant [-\dot p(t)-2m]\;[x-\ell(s)]\qquad\forall\, t\in[t-\delta,t],\
x\in (\ell(s),\ell(t)+M\delta].\eess
 Since $u\in C^\infty(Q_b)$ and $u>0$ in $Q_b$,  the above inequality implies that $(\ell(s),s)\not\in Q_b$. Hence, we must have $\ell(s)\leqslant b(s)$, for every $s\in[t-\delta,t]$.
  Consequently,   for each $s\in[t-\delta,t]$ and $x\in[\ell(s),b(t)+M\delta]$,
\bes\label{1b1}\qquad w(x,s) =p(s)-\int_{-\infty}^x u(y,s)dy = p(s)-\int_{\ell(s)}^x u(y,s)ds
\geqslant p(s) + \tfrac{1}{2}[\dot p(t)+2m][x-\ell(s)]^2.\ees

Now for any sufficiently small  positive $\vep$, consider the smooth function
\bess \psi_\vep(x,s)= p(s) + \tfrac{1}{2}[\dot p(t)+m]\; [x-\ell_\vep(s)]^2,\qquad \ell_\vep(s):=\ell(s)+[\vep-(t-s)].\eess
 We compare $w$ and $\psi_\vep$ in the set
\bess Q_\vep:=\{ (x,s)\;|\; s\in(t-\vep,t], x\in(\ell(s), \ell(s)+\sqrt\vep)\}.\eess
We claim that the maximum of $\psi_\vep-w$ on $\overline{Q_\vep}$ is positive and is attained at some point $(x_\vep,t_\vep)\in Q_\vep$.

First of all,  $(\ell_\vep(t),t)\in Q_\vep$ and
$\psi_\vep(\ell_\vep(t),t)-w( \ell_\vep(t),t) =p(t)-w(b(t)+\vep,t)>0.$

 Next we show that $\psi_\vep-w\leqslant 0$ on the parabolic boundary of $Q_\vep$. On the left lateral  boundary of  $Q_\vep$, $x=\ell(s)\leqslant b(s)$, so  $w(x,s)=p(s)$ and
$\psi_\vep-w\leqslant 0$.
For the remainder of the parabolic boundary  we can use   (\ref{1b1}) to derive
\bess \psi_\vep-w &\leqslant &
  \tfrac{1}{2}[\dot p(t)+m]\; [x-\ell_\vep(s)]^2-\tfrac{1}{2}[\dot p(t)+2m] \;[x-\ell(s)]^2
  \\  &=& - \tfrac{m}{2} [x-\ell(s)]^2 -\tfrac{1}{2}[\dot p(t)+m]\;[  2x- \ell(s)- \ell_\vep(s)]\;[\ell_\vep(s)-\ell(s)]  .\eess
On the lower side of the parabolic boundary of $Q_\vep$, $s=t-\vep$ so $\ell(s)=\ell_\vep(s)$ and
$\psi_\vep-w \leqslant0$.
Finally, on the right lateral boundary of $Q_\vep$, we have  $x=\ell(s)+\sqrt\vep$ and $0\leqslant \ell_\vep(s)-\ell(s)\leqslant \vep$  so
\bess \psi_\vep-w
 \leqslant - m \vep + |\dot p(t)+m|\vep^{3/2}.   \eess
Thus, $\psi_\vep-w\leqslant 0$ on the parabolic boundary of $Q_\vep$ if $\vep$ is sufficiently small.
Consequently, there exists $(x_\vep,t_\vep)\in Q_\vep$ such that $\psi_\vep-w$ attains at $(x_\vep,t_\vep)$ the positive maximum of $\psi_\vep-w$ over $\overline{Q_\vep}$.

We claim that $x_\vep>b(t_\vep)$. Indeed, if $x_\vep\leqslant b(t_\vep)$, then $\psi_\vep(x_\vep,t_\vep)-w(x_\vep,t_\vep)=\psi_\vep(x_\vep,t_\vep)-p(t_\vep)\leqslant 0$ which contradicts the fact that the maximum of $\psi-w$ on $\bar Q_\vep$ is positive. Thus $x_\vep>b(t_\vep)$. It then implies that $w$ is smooth in a neighborhood of $(x_\vep,t_\vep)$. Consequently,
$\cL\psi_\vep(x_\vep,t_\vep) \geqslant \cL w(x_\vep,t_\vep)=0$. However, 
\bess \cL \psi_\vep(x_\vep,t_\vep) &=& \dot p(t_\vep) +[ \dot p(t)+m ]\;[(x_\vep-\ell_\vep(t_\vep))(\mu(x_\vep,t_\vep)-\dot\ell(t_\vep)-1)-1 ]
\\  &\leqslant & -m +  [\dot p(t_\vep)-\dot p(t) ] + |\dot p(t)+m| \sqrt\vep \; (|\mu(x_\vep,t_\vep)|+M+1) .\eess
The last quantity is negative if we take $\vep$ sufficiently small. Thus we obtain a contradiction, and the second inequality in (\ref{1a}) holds. This completes the proof. \qed

\bigskip

\section{The Traditional Hodograph Transformation} \label{ClassicalHodographSection}

The  hodograph transformation  considers the inverse, $x=X(z,t)$, of the function $z=u(x,t)$, so that for each fixed $z$, the curve $x=X(z,t)$ is the $z$-level set of $u$. A bootstrapping procedure is applied. By beginning with a weak regularity assumption 
on the free boundary $b(t) = X(0,t)$, and applying the regularity theory for the partial differential equation satisfied by $X$, 
one can obtain higher-order regularity of $b$. In this section, we present two results derived using this traditional approach, one for the viscosity solution and the other for the  classical solution of the free boundary problem.

 \bigskip

 \begin{proposition} Let $b$ be the solution of the inverse problem associated with
 $p\in P_0$. Assume that for some interval $I=(t_{1},t_{2})$, $b\in C^{1}(I)$, and $p\in C^{\alpha+1/2}(I)$, 
 where $\alpha$ is not an integer. Then $b\in C^{\alpha}(I)$.
\end{proposition}
\begin{proof}
As $b$ is already  $C^1$, we need only consider the case $\alpha>1$, so $p$ is continuously  differentiable.
Since $b\in C^1(I)$, by working on the function
$U(y,t):=u(b(t)+y,t)$ on the fixed domain $[0,\infty)\times I$ with the boundary condition
 $U(0,\cdot)=0$ and then translating the regularity of $U$  back to $u(x,t)=U(x-b(t),t)$, one can show that for every $\gamma\in(0,1)$,  $$u\in C^\infty (Q_b)\cap C^{1+\gamma,(1+\gamma)/2}(\{(x,t)\;|\; t\in I, x\in[b(t),\infty)\}).$$

 In addition, since $b\in C^1(I)$ and  $u\geqslant 0$, the Hopf Lemma (\citet{ProtterWeinberger}, 
 Theorem 3.3, pages 170-171)
 implies that $u_x(b(t),t)>0$. Consequently, as $u_x\in C^{\gamma,\gamma/2}(\{(x,t)\;|\; t\in I, x\in[b(t),\infty)\})$,  $b\in C^1(I)$, and $p\in C^{\alpha+1/2}(I)$ with $\alpha>1$,   we can use
  Lemma \ref{le1a} and Remark \ref{re1a}  to derive that $2\dot p(t)=-\sigma^2 u_x|_{x=b(t)+}<0$.

 Once we know the continuity of $u_x$ and the positivity of $u_x(b(t),t)$,  we can define the inverse $x=X(z,t)$ of $z=u(x,t)$
 for $t\in I$ and $x\in[b(t),b_1(t))$ where $b_1(t)=\min\{x>b(t)\;|\; u_x(x,t)=0\}$.
 Then implicit differentiation gives:\footnote{Recall that we take $\sigma\equiv\sqrt{2}$ to simplify the exposition.}
  $$ X\in C^\infty, \quad X_t=X_z^{-2}\,X_{zz}+[z\mu(X,t)]_z,\quad \hbox{ \ in \ }
 \{(z,t)\;|\; t\in I, z\in (0,u(b_1(t),t))\}.$$ Also, $X_z\in C^{\gamma,\gamma/2}(D)$ where
 $D=\{(z,t)\;|\; t\in I, z\in [0,u(b_1(t),t))\}$ and
 $X_z(0,t)=-1/\dot p(t)$ for  $t\in I$. It then follows from the 
 local regularity theory for parabolic equations (see \cite{LSU}, Theorem IV.5.3, pages 320-322)  that when $p\in C^{\alpha+1/2}(I)$ where $\alpha>1$ is not an integer,  since $X_z(0,\cdot)=-1/\dot p(\cdot)\in C^{\alpha-1/2}(I)$ we have $X\in C^{2\alpha, \alpha}(D)$. Consequently,  $b=X(0,\cdot)\in C^{\alpha}(I)$.
This completes the proof. 
\end{proof}

\bigskip

\begin{proposition}  Suppose $\dot p<0$ on $(0,\infty)$ and $p\in C^{\alpha+1/2}((0,\infty))$ where $\alpha\geqslant 1/2$ is not an integer. Assume that $(b,u)$ is a classical solution of the free
boundary problem {\rm(\ref{1u})} satisfying
\bes\label{4.fbc} \lim_{x>b(s), (x,s)\to (b(t),t)} u_x(x,s)= \frac{-2 \dot p(t)}{\sigma^2(b(t),t)}\quad\forall\, t> 0.\ees
Then $b\in C^\alpha((0,\infty))$.
\end{proposition}
\bigskip

{\it Proof.}  Let $[\delta,T]\subset(0,\infty)$ be arbitrarily fixed. The conditions (\ref{4.fbc}) and $\dot p<0$  imply that  the set  $\{(b(t),t)|\; t\in[\delta,T]\}$ is compact, so that  there exists $\delta_1>0$ such that
$u_x$ is bounded and uniformly positive  in $D=\{(x,t)\;|\; t\in[\delta,T], x\in (b(t),b(t)+\delta_1]\}.$ Consequently, the inverse $x=X(z,t)$ of $z=u(x,t)$ is well-defined for $(x,t)\in D$. Setting $\vep=\min_{t\in[\delta,T]} u(b(t)+\delta_1,t)$ we have that $X\in C^\infty((0,\vep]\times[\delta,T])$,  $X_z$ is uniformly positive and bounded in $(0,\vep]\times[\delta,T]$ and
\bess X_z(0,t):= \lim_{z\searrow 0,s\to t} X_z(z,s)= \lim_{x>b(s),(x,s)\to (b(t),t)}\frac{1}{u_x(x,s)}=-\frac{1}{\dot p(t)}.\eess
Since $X$ satisfies $X_t=X_{z}^{-2} X_{zz}+[z\mu ]_z$ in $(0,\vep]\times[\delta,T]$, as above
local regularity then implies that $X$ defined on $(0,\vep]\times[\delta,T]$ can be extended onto $[0,\vep]\times[\delta,T]$ such that $X\in C^{2\alpha,\alpha}([0,\vep]\times(\delta,T])$.  Hence $b=\lim_{z\to 0}X(z,\cdot)=X(0,\cdot)\in C^{\alpha}((\delta,T])$. Sending $\delta\searrow 0$ and $T\to\infty$ we conclude that $b\in C^\alpha((0,\infty))$.\qed

\medskip


\section{Proof of The Main Result}
In this section, we prove our main result, Theorem \ref{mainth}. This provides regularity of the viscosity solution of the 
inverse problem without the a priori regularity assumptions used in the previous section. We begin with a technical 
result - a generalization of Hopf's Lemma in the one-dimensional case that is needed in our later arguments. We then 
introduce the scaling function $K$, and the new hodograph transformation for the scaled function $v = u/K$. Finally, 
we prove Theorem~\ref{mainth} by analyzing a family of perturbed equations related to the PDE satisfied by $X(z,t)$.

\subsection{A Generalization of Hopf's Lemma in the One-Dimensional Case}
In order to apply the weak formulation of the free boundary condition, we need the following extension of the classical Hopf Lemma.

\begin{lemma}[{\bf Generalized Hopf's Lemma}]\label{lea1} Let $\cL=\partial_t-a\partial_{xx}+c\; \partial_x+d$ where $a$, $c$ and $d$ are bounded functions and $\inf a>0$. Assume that $\cL \phi \geqslant 0$ in $Q:=\{(x,t)\;|\; 0<t<T, l(t)<x<r(t)\}$ where
$l$ and $r$ are Lipschitz continuous functions. Also assume that $\phi>0$ in $Q$. Then for every $\delta\in(0,T)$, there exists $\eta>0$ such that
\bess  \phi(x,t) \geqslant \eta [x-l(t)][
r(t)-x]\quad\forall\, x\in [l(t),r(t)], t\in [\delta,T].\eess
  Moreover, if in addition  $\phi(l(T),T)=0$, then $\phi_x(l(T),T)\geqslant \eta [r(T)-l(T)]$; similarly, if $\phi(r(T),T)=0$, then $\phi_x(r(T),T)\leqslant-\eta [r(T)-l(T)].$ 
\end{lemma}\bigskip

{\it Proof.} Without loss of generality, we can assume that $\phi(x,0)>0$ for all $x\in(l(0),r(0))$.
By modifying $r$ and $l$ in $[0,\delta]$ by
 $\tilde l(t)=l(t)+\vep [\delta-t] $ and $\tilde r(t)=r(t)-\vep [\delta-t]$ we can further assume that
 $\phi(x,0)$ is uniformly  positive on $[l(0),r(0)]$; i.e., there exists $\hat\eta>0$ such that
 $\phi(x,0)> \hat\eta$ for all $x\in[l(0),r(0)]$.
  Furthermore, by approximating $r$ by smooth functions from below and $l$ by smooth functions  from above, with the same Lipschitz constants of $r$ and $l$, we can assume that both $r$ and $l$ are smooth functions.

  Now for  large positive constants $M$ and $L$ to be determined, consider the function
 \bess \psi(x,t)&=& \hat\eta e^{-Mt-L (x-[r(t)+l(t)]/2)^2/2} \sin\frac{\pi [x-l(t)]}{r(t)-l(t)}.\eess
 Direct calculation gives
 \bess \cL \psi(x,t)& =&\hat\eta e^{-Mt-L[x-(r+l)/2]^2/2}\Big\{ A \cos \frac{\pi (x-l)}{r-l}
     +B\ \sin \frac{\pi (x-l)}{r-l}
 \Big\}
\eess
where
\bess A &=& \frac{\pi(x-r)l'+\pi (l-x)r'}{(r-l)^2}+2a L \Big(x-\frac{r+l}{2}\Big)\frac{\pi}{r-l} +\frac{c\;\pi }{r-l},
\\ B&=& \Big\{-M + L \Big(x-\frac{r+l}2\Big)\frac{r'+l'}{2}\Big\} +
a\Big\{L -L^2\Big(x-\frac{l+r}2\Big)^2 +\frac{\pi^2}{(r-l)^2}\Big\}- c L \Big(x-\frac{r+l}2\Big)+d.\eess
First taking $L= 2 \max_{[0,T]} \{(|r'|+|l'|+\|c\|_{L^\infty})/(r-l)\}/\inf a$ and then taking a suitably  large $M$ we see that $\cL\psi\leqslant 0$.
Thus, by comparison, we have $\phi(x,t)\geqslant \psi(x,t)$ in $Q$  and obtain the assertion of the Lemma.\qed

\bigskip

\subsection{Scaling the Survival Density}
In making comparison arguments in the proof of Theorem~\ref{mainth}, we would like to have that 
constant functions are solutions of the partial differential equation under consideration. Unfortunately, 
while this is true for the operator $\cL=\partial_t-\partial^2_{xx}+\mu\partial_x$, it may fail for $\cL_1=\partial_t- \partial_{xx}^2 \sigma^2 +\partial_x\mu$. To resolve this technical difficulty, we introduce a suitable scaling of the function $u$. 

To this end, let $K=K(x,t)$ be the bounded solution of the initial value problem:
\bes\label{K} \cL_1 K:= \partial_t K -  \partial_{xx} K + \partial_x (\mu K)=0\hbox{ \ in \ }\BbbR\times(0,\infty),\quad K(\cdot,0)=1\hbox{ \ on \ }\BbbR\times\{0\}.\ees
Then $K$ is smooth,  uniformly  positive, and  bounded in $\BbbR\times[0,T]$ for any $T>0$.
Now write
\bess u(x,t)= K(x,t)\; v(x,t) .\eess
It is easy to verify that
\bess \cL_1 u =K\; \cL_2 v,\eess
where \bess \cL_2:= \partial_t- \partial_{xx} + \nu\;\partial_x,
 \qquad \nu(x,t):= \mu(x,t) - \partial_x \log K^2(x,t). \eess
In particular, note that constant functions $c$ satisfy the equation $\cL_{2} c = 0$.

\subsection{The New Hodograph Transformation}
As we have seen, the traditional hodograph transformation is defined as the inverse of $z=u(x,t)$. In order to 
work with the scaled survival density and the operator $\cL_{2}$ 
introduced above, we instead define $x=X(z,t)$ as the inverse of $z=v(x,t)$:
\bess      z = v(X(z,t),t).\eess
If $z_0=v(x_0,t_0)>0$ and $v_x(x_0,t_0)\not=0$, then by the Implicit Function Theorem,
 locally the above equation defines a unique smooth $X$
for $(z,t)$ near $(z_0,t_0)$  that satisfies  $X(z_0,t_0)=x_0$.  In addition,
 by implicit differentiation,
\bess v_x =\frac{1}{ X_z},\qquad  v_t = -\frac{X_t}{X_z},\qquad v_{xx} = -\frac{X_{zz}}{X_v^3}.\eess
Thus, $\cL_2 v=0$ implies that $X=X(z,t)$ satisfies the following quasi-linear partial differential equation of parabolic type:
\bess X_t =X_z^{-2} \; X_{zz} + \nu(X,t).\eess
The free boundary is given by $b(t)=X(0,t)$ on which we can derive the boundary condition as follows. Since $u(b(t),t)=0$  and $u_x(b(t),t)= -\dot p(t)$, we see that $v_x(b(t),t)= -\dot p(t)/K(b(t),t)$. Hence, we have the non-linear boundary condition
\bess   \dot p(t)\; X_z(0,t) + K(X(0,t),t)=0.\eess
 This is a standard boundary condition for the quasi-linear parabolic equation for $X$.

\bigskip

\subsection{The Basic Assumption}
Let $p\in P_0$ and  $(b,w,u)$ be the unique viscosity solution of the inverse problem associated with $p$. To prove Theorem \ref{mainth}, we need only establish the assertions of the theorem  in a finite time interval $[0,T]$ for any fixed positive $T$. Hence, in the sequel, we assume,  for some $T>0$ and $x_0>0$, that
\bes \label{5.1} && p\in  C^{1}([0,T]), \ \dot p<0\hbox{  on  }[0,T] ,
 \quad \label{5.2}   \  u_0\in C^{1}([0,x_0]),  \ u'_0>0 \hbox{ on }[0,x_0],\ u_0=0\hbox{  on  }(-\infty,0].\ees
These conditions are satisfied by the general assumptions made in  Theorem \ref{mainth} if (i) in (\ref{1.up}) is imposed. If instead of (i), condition (ii) in (\ref{1.up}) is imposed,
we can apply the third assertion of Lemma \ref{le2.1} to shift the initial time by  considering first the solution   $(\hat b,\hat u):=(b(s+\cdot)-b(s),u(\cdot+b(s),\cdot+s))$ for $s$ at which $w_{xt}(\cdot,s)\in L^2(\BbbR)$, and then sending $s\searrow 0$.

Under (\ref{5.1}),
we will show that $b\in C^{1/2}((0,T])$ and that $(b,u)$ is a classical solution of the free boundary problem (\ref{1u}) on $\BbbR\times[0,T]$.

\medskip

First of all, from Lemma \ref{le2.1} (1) and (2),  we see that
\bess  u\in C(\BbbR\times(0,T]\cup ((-\infty,x_0)\times\{0\})),\quad  b\in C([0,T]),\qquad b(0)=0.\eess\medskip


\subsection{ The Level Sets.}
The hodograph transformation we are going to use is based on the inverse,  $x=X(z,t)$,  of $z=v(x,t)$ where $v=u/K$.
 That is, for each $z$, $X(z,\cdot)$ is the  $z$-level set of $v$.
 For $X$ to be well-defined, we need to consider $v$ in the set where $v_x$ is positive.\medskip

 We begin by investigating the initial value $X_0=X(\cdot,0)$ specified by
 $z=u_0(X_0(z)). $  Assume that $u_0$ satisfies (\ref{5.2}).
  Then the function $z=u_0(x), x\in[0,x_0]$, admits a unique inverse,  $x=X_0(z)$, which satisfies:
\bess   X_0 \in C^1([0,u_0(x_0)]), \qquad     z = u_0( X_0(z)),\qquad   X'_{0}(z)=\frac{1}{ u'_{0}(X_0(z))},
\  \ \forall\, z\in[0,u_0(x_0)].\eess

\medskip

Next we consider $X(\cdot,t)$ for small $t$. Fix  $\vep\in(0,u_0(x_0))$. There exists
$\delta_1>0$ such that $X_0(\vep)+2\delta_1<x_0$ and  $b(s)< X_0(\vep)-2\delta_1$ for all $s\in[0,\delta_1]$.
Also, since $u\in C(\BbbR\times (0,\infty) \cup (-\infty,x_0)\times\{0\})$,
there exists $\delta_2\in(0,\delta_1]$ such that $u<\vep$ on $(-\infty,X_0(\vep)-\delta_1]\times[0,\delta_2]$ and  $u>\vep$  on $\{X_0(\vep)+\delta_1\}\times[0,\delta_2]$.  Finally, from $\cL_2 v=0$ in
$[X_0(\vep)-2\delta_1,X_0(\vep)+2\delta_1]\times(0,\delta_2]$  we see that
$v_x$ is continuous and uniformly positive and $v_t=O(t^{-1/2})$  on $[X_0(\vep)-\delta_1,X_0(\vep)+\delta_1]\times[0,\delta_3]$ for some $\delta_3\in(0,\delta_2]$. Hence for each $t\in[0,\delta_3]$,
 the equation $\vep=v(x,t)$, for $x\in [X_0(\vep)-\delta_1,X_0(\vep)+\delta_1]$, admits a unique solution which we denoted by $x=X(\vep,t)$. This solution has the property that
 $X(\vep,t)=\min\{x>b(t)\;|\; v(x,t)=\vep\}$. In addition,
$X_t(\vep,t)=-v_t/v_x=O(t^{-1/2})$. Hence,  $X(\vep,\cdot)\in C^\infty((0,\delta_3])\cap C^{1/2}([0,\delta_3])$.

\medskip

  Now we extend  the inverse,  $x=X(z,t)$, of $z=v(x,t)$ to   $(z,t)\in (0,z_0)\times[0,T]$ where  $T$ is an arbitrarily fixed positive  constant and $z_0$ is a small positive constant
 that depends on $T$ and on the viscosity solution $u$. More precisely, we prove the following.
\bigskip

\begin{lemma} Let $p\in P_0$ be given and $(b,u)$ be the unique viscosity solution of the inverse problem associated with $p$. Assume that  {\rm(\ref{5.1})}  holds.  Let $K$ be defined in {\rm(\ref{K})} and $v:=u/K.$

 Then  there exists $z_0>0$ such that the function
\bes\label{zX} X(z, t):= \min\{ x \geqslant b(t)\;|\; v(x,t) =z\} \qquad\forall\, (z,t)\in [0,z_0]\times[0,T]\ees
is well-defined and satisfies
\bess   & \displaystyle z = v(X(z,t),t),\qquad
  v_x(X(z,t),t)>0,\quad  \forall\, (z,t)\in (0,z_0]\times[0,T],
\\ & \displaystyle X \in  C^\infty((0,z_0)\times(0,T])\cap C^{1,1/2}((0,z_0]\times[0,T]).
 \eess
\end{lemma}

\begin{proof}
We believe that the idea behind this proof may have appeared previously in the literature, but are not aware of a precise reference. For completeness, 
and possible other applications, we present a full proof.
We consider the level sets of $v$ in $\BbbR\times(0,\infty)$.
We call $z \in\BbbR$  a critical value of $v$ if there exists $(x,t)\in\BbbR\times(0,\infty)$ such that $v(x,t)=z$ and either $v_t(x,t)=v_x(x,t)=0$ or $v$ is not differentiable at $(x,t)$. Hence, if $z>0$ is not a critical value, then by the Implicit Function Theorem,  the level set $\{ (x,t)\in \BbbR\times(0,\infty)\;|\;  v(x,t)=z\}$  consists of smooth curves, each of which  either does not have boundary (i.e. lies completely inside  $Q_b$)   or has boundary  on $(0,\infty)\times\{0\}$.
Since $v\in C(\BbbR\times(0,\infty))\cap C^\infty(Q_b)$ and $v(x,t)=0$ for all $x\leqslant b(t)$,  by Sard's 
Theorem (e.g. \citet{GuilleminPollack}, pages 39-45) the set of all critical values
  of $v$ has measure zero.

\medskip
As before, we denote by $x=X_0(z)$ the inverse of $z=u_0(x), x\in[0,x_0]$.  \medskip

 Since $b(0)=0$ and $b\in C([0,T])$, we can find a smooth function $\ell \in C^\infty([0,T])$  such that $\ell(0)=x_0$ and $\ell(t)>b(t)$ for all $t\in[0,T]$. We define
 \bess \Gamma:= \{(\ell(t),t)\;|\; t\in[0,T]\}, \qquad z_0:=\min_{\Gamma} v= \min_{t\in[0,T]}v(\ell(t),t).\eess
  Then $X$ in (\ref{zX}) is well-defined and $X(z,t)\in (b(t), \ell(t)]$ for all $t\in[0,T]$ and $z\in(0,z_0]$. In addition, $X(0,t)=b(t)$ for all $t\in [0,T]$ and  $X(z,0)=X_0(z)$ for all $z\in(0,z_0]$. \medskip

 Next, let $\vep\in(0,z_0)$ be a non-critical value of $v$ and let $\gamma_\vep$ be the smooth curve in $\{(x,t)\;|t>0, v(x,t)=\vep\}$ that connects  to $(X_0(\vep),0)$.
 We first claim that $\gamma_\vep$ is a simple curve, i.e., it cannot form a loop. Indeed, if it forms a loop, then the loop is in $Q_b$ so the differential equation $v_t=v_{xx} -\nu v_x$ in $Q_b$ implies that $v\equiv \vep$ inside the loop which is impossible, since $\vep$ is not a critical value. Hence, $\gamma_\vep$ is a simple curve.
We parameterize $\gamma_\vep$ by its arc-length parameter, $s$, in the form $(x,t)=(x(s),t(s)), s\in(0,l)$, with $(x(0+),t(0+))=(X_0(\vep),0)$.
It is not difficult to show that $\lim_{|x|+t\to\infty} |v(x,t)|=0$, so $\gamma_\vep$ stays in a bounded region and we must have $l<\infty$ and $t(l-)=0, x(l-)>x_0$. In addition,
by the earlier discussion of $X(\vep,t)$ for small positive $t$ we see that $t'(s)>0$ and $x(s)=X(\vep,t(s))$ for all small positive $s$.\medskip

 When $t(s)>0$, we can differentiate $v(x(s),t(s))=\vep$ to obtain
 \bes\label{xs1}& v_x(x(s),t(s))\; x'(s) + v_t(x(s),t(s))\; t'(s) =0,
 \\   \label{xs2}& v_{xx}\; x'{}^2 + 2 v_{xt} \;x'\; t' + v_{tt}\; t'{}^2 + v_{x}\; x''+v_t\; t''|^{x=x(s)}_{t=t(s)}=0.\ees

  Now we define $l_0=\sup\{ s\in(0,l)\;|\; t'>0 \hbox{ \ in \ }(0,s]\}$. Since $t(l)=0$,  we must have $l_0\in(0,l)$ and  $t'(l_0)=0$. Consequently,
 $x'(l_0)^2 =1-t'(l_0)^2=1$.
 Evaluating (\ref{xs1}) at $s=l_0$, we obtain $v_x(x(l_0),t(l_0))=0$.
 Consequently,  evaluating (\ref{xs2}) at $s=l_0$ and using $v_t= v_{xx}-\nu v_x$ we obtain
 $ v_t(x(l_0),t(l_0)) [1+t''(l_0)]=0$.  Since $\vep$ is not a critical value of $v$, we must have $v_t(x(l_0),t(l_0))\not=0$, so that $t''(l_0)=-1$.
 This implies that $t'(s)<0$ for all $s$ bigger than and close to $l_0$.\medskip

 Next we define $l_1= \sup\{ s \in(l_0,l)\;|\; t'<0 \hbox{ \ in \ } (l_0,s]\}.$
 We claim that $l_1=l$. Indeed, suppose $l_1<l$. Then we must have $t(l_1)>0$ and $t'(l_1)=0$.
As above, first evaluating (\ref{xs1}) at $s=l_1$ we obtain $v_x(x(l_1),t(l_1))=0$ and then evaluating (\ref{xs2}) at $s=l_1$ and using $v_t= v_{xx}-\nu v_x$ we get that $t''(l_1)=-1$. However, this is impossible since $t'(l_1)=0$ and $t'(s)<0$ for all $s\in(l_0,l_1).$ Thus, we must have $l_1=l$.
In summary, we have
\bess t(0+)=0,\quad t'(s)>0\ \forall\, s\in(0,l_0),\quad t'(l_0)=0,\  t''(l_0)=-1,\quad
   t'(s)<0\ \forall\,s\in (l_0,l),\quad t(l-)=0. \eess

 Now we claim that  $t(l_0)>T$. Suppose not. Then $t(s)\leqslant T$ for all $s\in[0,l]$. Since $\min_{\Gamma} v=z_0>\vep$, we see that $\gamma_\vep$ cannot touch $\Gamma$. That is, $x(s)<\ell(t(s))$ for all $s\in[0,l]$. Thus, we must have $x(l)=X_0(\vep)$.
 However, this would imply $\gamma_\vep$ forms a loop which is impossible.
Hence, $t(l_0)>T$.

 Let $\hat l \in(0, l_0)$ be the number such that
 $t(\hat l)=T$.  Denote the inverse of $t=t(s)$, $s\in[0,\hat l],$ by $s=S(t)$.
  We can apply the Maximum Principle (\cite{ProtterWeinberger}, Theorem 3.2, recalling that $\cL_2 v\leqslant 0$ on $\BbbR\times(0,\infty)$)  for $v$ on the domain $\{(y,t)\;|\; t\in[0,T],   y\in (-\infty,x(S(t))]\}$ to conclude that  $v(y,t)<\vep$ for every $y<x(S(t))$ and $t\in[0,T]$.
  Hence, we must have $x(S(t))=\min\{x>b(t)\;|\; v(x,t)=\vep\}=X(\vep,t)$. Since
  $\vep$ is not a critical value and $t'>0$ in $(0,\hat l]$, we derive from (\ref{xs1}) that  $v_x(X(\vep,t),t)>0$ for every $t\in(0,T]$. Hence, $X(\vep,\cdot)\in C^\infty((0,T])\cap C^{1/2}([0,T])$.
\medskip

Finally,
let $z_1,z_2$ be any two non-critical values of $v$ that satisfy $0<z_1<z_2\leqslant z_0$.
Then for $i=1,2$,  $X(z_i,t):=\min\{ x>b(t)\;|\; v(x,t)=z_i\}$ is a smooth function  and
$v_x(X(z_i,t),t)> 0$ for $t\in[0,T]$.  Note that $v_x$ satisfies  the
equation $(v_x)_t =  (v_x)_{xx} +\nu (v_x)_x+\nu_x v_x$ in $Q_b$, the  Maximum Principle then implies that $v_x>0$ on $\{(x,t)\;|\; t\in [0,T],\; x\in [X(z_1,t),X(z_2,t)]\}$.
By the Implicit Function Theorem, we then know that $X \in C^\infty((z_1,z_2)\times(0,T])$. Finally,
sending $z_1\to 0$ and $z_2\to z_0$ along non-critical values of $v$,
 we conclude that $X\in C^\infty((0,z_0)\times(0,T]) \cap C^{1,1/2}((0,z_0)\times[0,T])$.   This completes the proof of the Lemma.\end{proof}

\begin{remark}{\sl Here we have used the Maximum Principle for $v_x$.  The Maximum Principle may not hold for $u_x$ since the equation for $u_x$ is $(u_x)_t=(u_x)_{xx} - \mu u_{xx} - 2\mu_x u_{x} -\mu_{xx} u$ where the non-homogeneous term $\mu_{xx}u$ may cause difficulties.} 
Of course, we also need to know  that $v\in C^\infty(Q_b)\cap C(\BbbR\times(0,T]\cup((-\infty,x_0)\times\{0\}))$ and $\cL_2 v\leqslant 0$ on $\BbbR\times(0,\infty)$  so that the Maximum Principle can be applied to $v-\vep$.\end{remark}

Using the same idea as in the proof,  one can derive the following which maybe useful in qualitative
and/or quantitative studies of the free boundary.

\begin{proposition} Let $(b,u)$ be the solution of the inverse problem associated with $p\in P_0$.  Assume that $p\in C^1([0,\infty))$ and $\dot p<0$ on $[0,\infty)$. Let $v=u/K$ where $K$ is defined in {\rm(\ref{K})}.

\begin{enumerate}\item
Denote by $N(t)$ the number of roots (without counting multiplicity)  of $v_x(\cdot,t)=0$ in $(b(t),\infty)$. Then $N(t)$ is a decreasing function. In particular, if $N(0)=1$, i.e. $u_0'$ changes sign only once in $(b(0),\infty)$, then
$N(t)\equiv 1$ for all $t>0$; that is, $v_x(\cdot,t)$ changes sign only once in $(b(t),\infty)$.

 \item Suppose $u_0$ is the Delta function (i.e. $\fX_0=0$ a.s.). Then there exists $X_1\in C([0,\infty))\cap C^\infty((0,\infty))$ such that
$X_1(0)=0$ and $v_x(\cdot,t)>0$ in $(b(t),X_1(t))$ and $v_x(\cdot,t)<0$ in $(X_1(t),\infty))$ for every $t>0$.
\end{enumerate}\end{proposition}
The first assertion can be proven by a variation of  our proof. The second assertion follows from the fact that the Delta  function can be approximated by a sequence of the bell-shaped  positive functions, each of which has only one local maximum and no local minimum; see \cite{CCJW}.   We omit the details.

\subsection{The Initial Boundary Value Problem}
In the sequel, $X$ is defined by (\ref{zX}) and $\vep\in(0,z_0)$ is a fixed small positive constant.
We consider the quasi-linear parabolic initial boundary value problem, for the unknown function $Y=Y(z,t):$
\bes\label{2a} \left\{ \begin{array}{ll} \displaystyle Y_t  = Y_z^{-2}\; Y_{zz}+ \nu(Y,t),\qquad & z\in (0,\vep), t\in(0,T], \medskip
\\     Y(z,0) = X_0(z), & z\in[0,\vep], \medskip
\\ \dot p(t)\, Y_z(0,t)+K(Y(0,t),t) = 0, & t\in(0,T],\medskip
\\ Y_z(\vep,t) =X_z(\vep,t), & t\in(0,T].
\end{array}\right.
\ees

We know from the theory of quasi-linear equations that this problem admits a unique classical solution for $\vep$ small 
enough, which is also smooth up to the boundary, so long as $X_{z}(\vep,t)$ for $t\in[0,T]$ and $X_{0}'(z)$ for $z\in[0,\vep]$ are uniformly positive.  
The main difficulty is to show that 
 $X=Y$.\footnote{Once we have $X=Y$, and $b=Y(0,\cdot)$, it follows from classical theory that $b$ is only $1/2$ less differentiable than $p$.} 
 We know that $X$ satifies all of~(\ref{2a}), except for the third equation (i.e., we don't know a priori that $\dot p(t)\, X_z(0,t)+K(X(0,t),t) = 0, t\in(0,T]$, because we don't have smoothness of $X$ up to the boundary).

To do this, we consider for each $h\in\BbbR$, the following initial boundary value problem,  for $Y^h=Y^h(z,t)$,
\bes \left\{ \begin{array}{ll} \displaystyle Y^h_t  = (Y^h_z)^{-2}\; Y^h_{zz}+
 \nu(Y^h,t),\qquad & z\in (0,\vep), t\in(0,T],\medskip
\\     Y^h(z,0) = X_0(z)+h, & z\in[0,\vep],\medskip
\\ \dot p^h(t)\; Y^h_z(0,t)+K(Y^h(0,t),t) = 0, & t\in(0,T],\medskip
\\ Y^h_z(\vep,t) =X_z(\vep,t), & t\in(0,T],
\end{array}\right.\label{2b}
\ees
where $\{p^h\}_{h\in\BbbR}$ is a family that has the following properties:
\bess p^h\in P_0\cap C^\infty([0,T]),\quad \dot p^h<0 \hbox{ on } [0,T],
\quad (\dot p-\dot p^h)h\geqslant 0, \quad \lim_{h\to0}\|p^h-p\|_{C^1([0,T])}=0.\eess

Note that if $\mu\equiv 0$ and $p^h\equiv p$, then $\nu\equiv 0$ and $K\equiv 1$ so problem (\ref{2b}) relates to  (\ref{2a}) by the simple translation $Y^h=Y+h$.  Here we shall consider the general case. 

\subsection{Well-Posedness of the Perturbed Problems}
We now show that for each $h\in\BbbR$, (\ref{2b}) admits a unique classical solution.
Since (\ref{2a}) and (\ref{2b}) belong to   the same type  of initial-boundary value problem, we state our result in terms of (\ref{2a}). The conclusion for (\ref{2b}) is analogous.
\bigskip

\begin{lemma} Let $K,\nu$ be  smooth and bounded  functions on $\BbbR\times[0,T]$. Assume that $K>0$ and $$X_0\in C^1([0,\vep]), X_0'>0,  \quad p\in C^1([0,T]), \dot p<0,\quad
 X_z(\vep,\cdot)\in C([0,T]),  X_{z}(\vep,\cdot)>0.$$
 Then problem {\rm(\ref{2a})} admits a unique classical solution that satisfies
 \bess  & Y\in C([0,\vep]\times[0,T])\cap C^\infty((0,\vep)\times(0,T]),\quad
  Y_z\in C([0,\vep]\times(0,T]),\\
 &\displaystyle 0< \inf_{[0,\vep]\times(0,T]} Y_z \leqslant \sup_{[0,\vep]\times(0,T]} Y_z<\infty.\eess

 If in addition $p\in C^{\alpha+1/2}((t_1,t_2])$ where $\alpha\geqslant1/2$  is not an integer and $(t_1,t_2]\subset(0,T]$, then $Y\in C^{2\alpha,\alpha}([0,\vep)\times(t_1,t_2])$ so that $Y(0,\cdot)\in C^\alpha((t_1,t_2])$. \medskip

 If $p\in C^{\alpha+1/2}([0,T])$, $X_0\in C^{2\alpha}([0,\vep))$, $\alpha\geqslant 1/2$ is not an integer,  and all compatibility conditions at $(0,0)$ up to order $\llbracket\alpha+1/2\rrbracket$ are satisfied, then $Y\in C^{2\alpha,\alpha}([0,\vep)\times[0,T])$ so $Y(0,\cdot)\in C^\alpha([0,T]).$
 \end{lemma}
\begin{proof}
According to the general theory of quasi-linear partial differential equations of parabolic type  (see \citet{LSU}), to show that (\ref{2a}) admits a unique classical solution, it suffices  to establish an a priori estimate for an upper bound and a positive lower bound for $Y_z$.
For this purpose, suppose we have a classical solution $Y$ of (\ref{2a}). Then using
a local analysis we have  $Y\in C^\infty((0,\vep)\times(0,T])$.  Set $H= Y_z$. Then we can differentiate the first two  equations in (\ref{2a}) to obtain:
\bess \left\{\begin{array}{ll} H_t = - \; (H^{-1})_{zz} + \nu_x(Y,t)\; H, \qquad &z\in (0,\vep), t\in (0,T],\medskip
 \\ H(z,0) = X_{0}'(z), & z\in[0,\vep],\medskip
\\ \displaystyle  \dot p(t)\; H(0,t) +K(Y(0,t),t)=0, & t\in(0,T],\medskip
\\  H(\vep,t) = X_z(\vep,t), & t\in(0,T].
\end{array}\right.
    \eess

Now denote
\bess M_1:= \frac{\|K\|_{L^\infty(\BbbR\times[0,T])}}{\min_{[0,T]} |\dot p|},
&& m_1:= \frac{\inf_{\BbbR\times[0,T]} K}{\max_{[0,T]} |\dot p|},
\\  M_2 := \max_{[0,x_0]} \frac{1}{u_0'(x)}, &&
m_2 := \min_{[0,x_0]} \frac{1}{u_0'(x)},
\\ M_3=\max_{t\in[0,T]} X_z(\vep,t),&&  m_3=\min_{t\in[0,T]} X_z(\vep,t),
\\  M=\max\{M_1,M_2,M_3\},&& m=\min\{m_1,m_2,m_3\},\eess
\bess
  k(t)= \|\nu_x(\cdot,t)\|_{L^\infty(\BbbR)}.\eess

Then by comparison, we have
\bess  m e^{-\int_0^t k(s)ds} \leqslant H(z,t)=Y_z(z,t) \leqslant M e^{\int_0^t k(s)ds}\quad\forall\, t\in[0,T], z\in[0,\vep].\eess

These a priori upper and lower bounds then imply that  (\ref{2a}) admits
 a unique classical solution.  The remaining assertions follow from the local and global regularity theory of parabolic equations \cite{LSU}. This completes the proof. 
 \end{proof}
 %
%
%

An an example, we demonstrate the derivation of   the first order compatibility condition:
\bess\frac{1}{- \dot p(0)} =   \lim_{t\searrow 0} Y_z(0,t)=\lim_{z\searrow 0} Y_z(z,0) = X_0'(0)= \frac{1}{u_0'(0)}. \eess

We remark that from our definition of  $X$, the compatibility of the initial and boundary data
at $(\vep,0)$ for  (\ref{2a}) is automatically satisfied. For (\ref{2b}), since $K(\cdot,0)\equiv 1$, the first order compatibility condition at $(\vep,0)$ is also satisfied, so $Y^h\in C^{1,1/2}([0,\vep]\cup[0,T]\setminus \{(0,0)\})$.

  \medskip

Finally,
to demonstrate continuous dependence, we   integrate  over $(0,\vep)\times\{t\},t\in(0,T],$ the difference of the differential equations in (\ref{2a}) and (\ref{2b}) multiplied by $Y-Y^h$ and use integration by parts to obtain
\bess && \frac12\frac{d}{dt} \int_0^\vep (Y-Y^h)^2 dz +\int_0^\vep \frac{(Y_z-Y_z^h)^2}{Y_zY_z^h} dz \\ && \quad
= \int_0^\vep (Y-Y^h)[\nu(Y,t)-\nu(Y^h,t)]dz
 +\Big[\frac{\dot p^h(t)}{K(Y^h(0,t),t)}-\frac{\dot p(t)}{K(Y(0,t),t)}\Big] [Y(0,t)-Y^h(0,t)]. \eess
 Upon using  the boundedness  of $Y_z$ and $Y_z^h$, Cauchy's inequality, and the Sobolev embedding
 \bess  \|\phi\|_{L^\infty([0,\vep])}^2 \leqslant \Big(\frac{1}{\vep}+\frac1\delta\Big)
  \int_0^\vep \phi^2(z)dz + \delta\int_0^\vep \phi^2_z(z)dz\quad\forall\,\delta>0,\eess
 and we find that there exists a positive  constant $C$ such that
  \bess \frac{d}{dt} \int_0^\vep (Y-Y^h)^2 dz  \leqslant C \int_0^\vep (Y-Y^h)^2 +C |\dot p^h(t)-\dot p(t)|^2\quad\forall\, t\in(0,T].\eess
Gronwall's inequality then yields the estimate
\bess  \max_{t\in[0,T]} \| Y(\cdot,t)-Y^h(\cdot,t)\|^2_{L^2((0,\vep))} \leqslant  C e^{CT} \Big\{  h^2
+ \int_0^T |\dot p^h-\dot p|^2 dt\Big\}.\eess
This estimate in turn implies the $C([0,T];L^2((0,\vep))$ convergence of $Y^h$ to $Y$ as $h\to 0$. By Sobolev embedding and the boundedness of $(Y-Y^h)_z$, this convergence also implies the $L^\infty$ convergence:
\bess \lim_{h\to 0} \|Y^h-Y\|_{C([0,\vep]\times[0,T])} =0.\eess
To show that  $b=Y(0,\cdot)$,  it suffices to show the following:
 \bigskip

 \begin{lemma} For every $h>0$,
 $Y^{-h}(0,\cdot)<b<Y^h(0,\cdot)$ and $Y^{-h}(\vep,\cdot)<X(\vep,\cdot)<Y^h(\vep,\cdot)$ on $[0,T]$. \end{lemma}
    The proof will be given in the  next three  subsections.
\bigskip

 \subsection{The Inverse Hodograph Transformation}
 To show that  $Y=X$ and  $b=Y(0,\cdot)$,  we define $z=v^h(x,t)$   as the inverse function of $x=Y^h(z,t)$:
 \bess x=Y^h(z,t) , z\in [0,\vep],t\in[0,T] \qquad\Longleftrightarrow \qquad z= v^h(x,t),
 \, x\in [Y^h(0,t), Y^h(\vep,t)],\, t\in [0,T].\eess
 Since $Y^h_z>0$, the inverse is well-defined. We record the key equation for future reference: 
 \bes\label{Yhx} x= Y^h(v^h(x,t),t) \qquad\forall\, t\in[0,T],\; x\in[Y^h(0,t),Y^h(\vep,t)]. \ees

 \medskip

By  implicit differentiation,
we find that
\bess v^h_x(x,t) &=& \frac{1}{Y^h_z(z,t)}\Big|_{z=v^h(x,t)},\medskip
\\  v^h_{xx}(x,t) &=& -\frac{Y^h_{zz}(z,t)}{Y^h_z(z,t)^3}\Big|_{z=v^h(x,t)},  \medskip
\\  v^h_t(x,t) &=& -\frac{Y^h_t(z,t)}{Y^h_z(z,t)} =-\frac{Y^h_{zz}(z,t)}{Y^h_z(z,t)^{3}} - \frac{\nu(Y^h(z,t),t)}{Y^h_z( z,t)}\Big|_{z=v^h(x,t)}
\\ &=&  v_{xx}^h(x,t) -\nu(x,t) v^h_x(x,t).\eess
Finally, setting $u^h(x,t)=K(x,t) v^h(x,t)$ we see that
\bess\cL_2 v^h=0,  \quad \cL_1 u^h(x,t)=0,\quad\forall\, t\in(0,T],  x\in (Y^h(0,t),Y^h(\vep,t)).\eess

\medskip

When $t=0$, using $Y^h(z,0)=X_0(z)+h$, we see from (\ref{Yhx}) that
$ x= Y^h(v^h(x,0),0)= X_0(v^h(x,0))+h.$
This implies that $x-h= X_0(v^h(x,0))$, so using $\hat x= X_0(u_0(\hat x))$ with $\hat x=x-h$ we obtain
$v^h(x,0)=u_0(x-h)$
or
\bess  u^h(x,0)=v^h(x,0)= u_0(x-h),\quad\forall\, x\in [Y^h(0,0),Y^h(\vep,0)]=[h,X_0(\vep)+h].\eess
Next, substituting  $x=Y^h(\vep,t)$ in (\ref{Yhx}) we obtain
\bess v^h(Y^h(\vep,t))=\vep=  v(X(\vep,t),t).\eess
The boundary condition $Y_z^h(\vep,t)=X_z(\vep,t)=1/v_x(X(\vep,t),t)$ then implies that
\bess  v_x^h(Y^h(\vep,t),t) = v_x(X(\vep,t),t).\eess
Finally,  substituting $x=Y^h(0,t)$ in (\ref{Yhx}) we have
\bess v^h(Y^h(0,t),t)=0= v(b((t),t).\eess
The boundary condition $\dot p^h(t)\;Y_z^h(0,t)+K(Y^h(0,t),t)=0$ then gives
\bess  v_x^h(Y^h(0,t),t) \; K(Y^h(0,t),t)= -\dot p^h(t).\eess

In summary, we see that $ v^h=v^h(x,t)$ has the following properties:
\bess \left\{\begin{array}{ll} \cL_2 v^h=0\quad
&\forall\, t\in(0,T], x\in (Y^h(0,t),Y^h(\vep,t)), \medskip
\\    v^h(x,0)= u_0(x-h),&\forall\, x\in[Y^h(0,0),Y^h(\vep,0)],\medskip
\\ v^h(Y^h(0,t),t)=0 = v( b(t),t), & \forall\; t\in[0,T],\medskip
\\ v^h_x(Y^h(0,t),t)\;K(Y^h(0,t),t) = -\dot p^h(t),  & \forall\; t\in[0,T],\medskip
\\   v^h(Y^h(\vep,t),t)=\vep =v(X(\vep,t),t),  & \forall\; t\in[0,T],\medskip
\\  v^h_x( Y^h(\vep,t),t) = v_x(X(\vep,t),t) , & \forall\; t\in[0,T].
\end{array}
\right.
\eess
Note that $[Y^h(0,0),Y^h(\vep,0)]=[h,X_0(\vep)+h].$

\bigskip

\subsection{The Proof that $b<Y^h(0,\cdot)$ and $X(\vep,\cdot)<Y^h(\vep,\cdot)$ for $h>0$} \
\medskip

Let $h>0$ be arbitrarily fixed. We define
\bess  T^*:=\sup\{ t\in[0,T]\;|\;  b<Y^h(0,\cdot)  \hbox{ \ and \ } X(\vep,\cdot)< Y^h(\vep,\cdot)\hbox{ \ in \ }[0,t]\}.\eess
We know that $X(\vep,\cdot),Y^h(0,\cdot)$ and $Y^h(\vep,\cdot)$ are all continuous and that
$b$ is upper-semi-continuous (\cite{CCCS2}). Also,
$Y^h(0,0)=h$, $b(0+)\leqslant 0$, and
$Y^h(\vep,0)=X_0(\vep,0)+h$. Thus, we have that $T^*$ is well-defined and $T^*\in(0,T]$.

We claim that $b<Y^h(0,\cdot)$ and $X(\vep,\cdot)<Y^h(\vep,\cdot)$ on $[0,T]$. Suppose the claim is not true.   Then  we must have  have
$ b(t)<Y^h(0,t)$ and $X(\vep,t)<Y^h(\vep,t)$ for all $t\in [0,T^*)$ and
\bess \hbox{either (i) \ } X(\vep,T^*)=Y^h(\vep,T^*),\qquad\hbox{ or (ii) \ } b(T^*)=Y^h(0,T^*).\eess
We shall show that neither of the above can happen. For this  we   compare $v$ and $v^h$ in the set
\bess Q &:=& \{ (x,t)\;|\; t\in (t_0,T^*], Y^h(0,t)<x< X(\vep,t)\},
\eess where \bess  t_0&:=&\inf\{t\in [0,T^*]\;|\;
Y^h(0,t)<X(\vep,t) \hbox{ \ in \ }[t,T^*]\}.\eess
Since $b(t)<Y^h(0,t)$ for all $t\in [0,T^*)$, we have  $\cL_2 v=\cL_2 v^h=0$ in $Q$. Also, on
  the left lateral parabolic boundary of $Q$, $x=Y^{h}(0,t)$, we have $v\geqslant 0=v^h$. On the right lateral boundary of $Q$, $x=X(\vep,t)$, we have $v^h\leqslant \vep=v$. Thus $v\geqslant v^h$ on the parabolic boundary of $Q$ if $t_0>0$.
  Finally, if $t_0=0$,  we have
  $v(x,0)=u_0(x) > u_0(x-h)=v^h(x,0)$ for all $x\in [Y^h(0,0),X(\vep,0)]=[h,X_0(\vep)]$.
  Thus, $v\geqslant v^h$ on the parabolic boundary of $Q$.  Since $h>0$, we cannot have $v\equiv v^h$, so the
Strong Maximum Principle implies that $v>v^h$ in $Q$.

Now consider case (i): $X(\vep,T^*)=Y^h(\vep,T^*)$. Then
as $X(\vep,\cdot)$ is smooth and
 $v(X(\vep,T^*),T^*)=\vep =v^h(Y^h(\vep,T^*),T^*)=v^h(X(\vep,T^*),T^*)$, the Hopf Lemma implies that
$v_x(X(\vep,T^*),t)< v^h_x(X(\vep,T^*),T^*)$. This  is impossible since  $v^h_x(X(\vep,T^*),T^*)=v_x^h(Y^h(\vep,T^*),T^*)=v_x(X(\vep,T^*),T^*)$. Thus case (i) cannot happen.

Next, we consider case (ii): $b(T^*)=Y^h(0,T^*)$.  Since $Y^h(0,\cdot)$ is smooth, the generalized Hopf Lemma \ref{lea1}  implies that there exist $\eta>0$ and $\delta>0$ such that
\bess v(x,s)-v^h(x,s) \geqslant [x-Y^h(0,s)]\;\eta,\quad\forall\, s\in[T^*-\delta, T^*], x\in
[Y^h(0,s), b(T^*)+\delta].\eess
However, since $v_x^h(Y^h(0,s),s)K(Y^h(0,s),t)=-\dot p^h(s)\geqslant -\dot p(s)$, the above inequality implies
\bess \varliminf_{x>Y^h(0,s), s\leqslant T^*, (x,s)\to (b(T^*),T^*)} \frac{ u(x,s)}{x-Y^h(0,s)} \geqslant - \dot p(T^*)+\eta K(b(T^*),T^*).\eess
But this contradicts the first inequality in (\ref{1a}) [with $\ell:=Y^h(0,\cdot)$] of Lemma \ref{le1a}. Hence,
case (ii) also cannot happen.

In conclusion, when $h>0$, we  must have
$b<Y^h(0,\cdot)$  and $X(\vep,\cdot)<Y^h(\vep,\cdot)$ on $[0,T]$.
\bigskip

\subsection{The Proof that $Y^h(0,\cdot)<b$ and $Y^h(\vep,\cdot)<X(\vep,\cdot)$ for $h<0$} \
 \medskip

 Here, we use the facts that $b(0)=0$ and  $b$ is continuous on $[0,T]$, proven  in \cite{CCCS2}, and repeated here 
 as Proposition~\ref{WeakRegularityProposition}. Also, we need the fact that $\cL_2 v\leqslant 0$ in $\BbbR\times(0,\infty)$.
 \medskip

Let $h<0$ be arbitrary. We define
\bess  T^*:=\sup\{ t\in[0,T]\;|\;  Y^h(0,\cdot)< b \hbox{ \ and \ }\quad Y^h(\vep,\cdot)< X(\vep,\cdot)\hbox{ \ on \ }[0,t]\}.\eess
Since $b,X(\vep,\cdot),Y^h(0,\cdot),Y^h(\vep,\cdot)$ are all continuous and $Y^h(0,0)=h<0=b(0)$ and
$Y^h(\vep,0)=X(\vep,0)+h$, we see that  $T^*$ is well-defined and $T^*\in(0,T]$.

\medskip
We  claim that $Y^h(0,\cdot)<b$ and $Y^h(\vep,\cdot)<X(\vep,\cdot)$ on $[0,T]$.
Suppose the claim is not true. Then
$ Y(0,t)<b(t)$ and $Y^h(\vep,t)<X(\vep,t)$ for all $t\in [0,T^*)$ and
\bess \hbox{either (i) \ } X(\vep,T^*)=Y^h(\vep,T^*),\qquad\hbox{ or (ii) \ } b(T^*)=Y^h(0,T^*).\eess
To show that none of the above can happen, we  compare  $v$ and $v^h$ as before in the set
\bess Q:= \{ (x,t)\;|\; t\in (0,T^*], Y^h(0,t)<x< Y^h(\vep,t)\}.\eess
Then, by the definition of $T^*$, we have $v(x,t)\leqslant v^h(x,t)$ on the parabolic boundary of $Q$ and $\cL_2 v^h=0\geqslant \cL_2 v$ in $Q$.
The Maximum Principle then implies that $v<v^h$ in $Q$.
\medskip

Now consider case (i): $X(\vep,T^*)=Y^h(\vep,T^*)$. Then
as $Y^h(\vep,\cdot)$ is smooth and
 $v(X(\vep,T^*),T^*)=\vep=v^h(X(\vep,T^*),T^*)$, the Hopf Lemma implies that
$v_x(X(\vep,T^*),T^*)> v^h_x(X(\vep,T^*),T^*)$. This  is impossible since  $v^h_x(X(\vep,T^*),T^*)=v_x(Y^h(\vep,T^*),T^*)=v_x(X(\vep,T^*),T^*)$. Thus case (i) cannot happen.
\medskip

Next we consider case (ii): $b(T^*)=Y^h(0,T^*)$.  Again, since $Y^h(0,\cdot)$ is smooth, the generalized Hopf Lemma \ref{lea1} implies that there exist $\eta>0$ and $\delta>0$ such that
\bess v^h(x,s)-v(x,s) \geqslant [x-Y^h(0,s)]\;\eta \quad\forall\, s\in[T^*-\delta, T^*], x\in
[Y^h(0,s), b(T^*)+\delta].\eess
However, since $v_x^h(Y^h(0,s),s) K(Y^h(0,s),s)=-\dot p^h(s) \leqslant -\dot p(s)$, the above inequality implies
\bess \varlimsup_{x>Y^h(0,s), s\leqslant T^*, (x,s)\to (b(T^*),T^*)} \frac{ u(x,s)}{x-Y^h(0,s)} \leqslant - \dot p(T^*)-\eta\, K(b(T^*),t).\eess
This contradicts the second inequality in (\ref{1a}) [with $\ell=Y^h(0,\cdot)$] of Lemma \ref{le1a}. Hence,
case (ii) also cannot happen. Thus, when $h<0$, we have
$Y^h(0,\cdot)<b$  and $Y^h(\vep,\cdot)<X(\vep,\cdot)$ on $[0,T]$.

\bigskip

\subsection{Proof of Theorem \ref{mainth}}
Once  we know that
 $Y^{-h}(0,t)<b(t)<Y^h(0,t)$ for $h>0$ and $t\in[0,T]$, we can send $h\to 0$ to conclude that  $b(t)=Y(0,t)$.
 Consequently, $b=Y(0,\cdot)\in C^{1/2}((0,T])$.

 As we know $\dot p\in C([0,T])$,  one can show that $Y_z\in C([0,\vep]\times(0,T])$, so
 \bess \lim_{z\searrow 0,s\to t} Y_z(z,s)= Y_z(0,t)=-\frac{1}{K(b(t),t)\dot p(t)}\quad\forall\, t\in(0,T].\eess
 Using $X=Y$ and $u_x(Y(z,t),t)=1/[K(Y(z,t),t) Y_z(z,t))]$, we then obtain
 \bess u_{x}(b(t)+,t):=\lim_{x>b(s), (x,s)\to (b(t),t)} u_x(x,s)= -\dot p(t)\qquad \forall\, t\in(0,T].\eess
 Thus, $(b,u)$ is a classical solution of  (\ref{1u}) on $\BbbR\times[0,T]$. If, in addition, $p\in C^{\alpha+1/2}((t_1,t_2))$ for some $\alpha>1/2$ that is not an integer, then by local regularity, $b=Y(0,\cdot) \in C^{\alpha}((t_1,t_2))$. If we further have $u_0\in C^{2\alpha}([0,x_0]), p\in C^{\alpha+1/2}([0,T])$ for some $\alpha\geqslant 1/2$ that is not an integer, and all compatibility conditions up to the order $\llbracket\alpha+1/2\rrbracket$ are satisfied, then  $Y\in C^{2\alpha,\alpha}([0,\vep]\times[0,T])$ so
 $b=Y(0,\cdot)\in C^{\alpha}([0,T])$.

 Finally, upon noting that $T$ can be  arbitrarily large, we also obtain the assertion  of Theorem \ref{mainth}, which completes 
 the proof of this theorem.

\section{Conclusion}
In earlier work, we studied the inverse first-passage problem for a one-dimensional diffusion process by relating it to a variational inequality. We investigated existence and uniqueness, as well as the asymptotic behaviour of the boundary for small times, 
and weak regularity of the boundary. In this paper, we studied higher-order regularity of the free boundary in the inverse 
first-passage problem. The main tool used was the hodograph transformation. The traditional approach to the transformation begins 
with some a priori regularity assumptions, and then uses a bootstrap argument to obtain higher regularity. We presented 
the results of this approach, but then went further, studying the regularity of the free boundary under weaker assumptions. In order 
to do so, we needed to perform the hodograph transformation on a carefully chosen scaling of the survival density, and to analyze the 
behaviour of a related family of quasi-linear parabolic equations. 
We expect that the method presented here can be applied to other parabolic obstacle problems.


\bibliographystyle{plainnat}
\bibliography{References}


\end{document}